\documentclass[11pt]{amsart}
\usepackage{amsmath, amssymb, amsthm}
\usepackage[margin=1in]{geometry}
\usepackage[hypertexnames=false]{hyperref}
\hypersetup{hidelinks}
\usepackage{booktabs}
\usepackage{enumitem}

\theoremstyle{plain}
\newtheorem{theorem}{Theorem}
\newtheorem{corollary}[theorem]{Corollary}
\newtheorem{proposition}[theorem]{Proposition}
\newtheorem{lemma}[theorem]{Lemma}
\newtheorem*{conjecture-fwd}{Conjecture (strict form)}

\theoremstyle{definition}
\newtheorem{conjecture}[theorem]{Conjecture}

\title[Bounded-box reductions for unitary perfect numbers]{Bounded-box reductions in the Subbarao--Warren problem for unitary perfect numbers}
\author[T. Maciejewski]{Tom Maciejewski}
\address{Independent Researcher; \texttt{etm8@njit.edu}}
\thanks{ORCID: 0009-0004-5702-5302.}
\date{May 2026}


\begin{document}

\begin{abstract}
A unitary perfect number is a positive integer $n$ satisfying
$\sigma^*(n)=2n$.  Only five examples are known.  We revisit the
Subbarao--Warren problem by keeping the seed factor $2^a+1$ explicit in
the full balance
\[
(2^a+1)\prod_i(p_i^{e_i}+1)=2^{a+1}\prod_i p_i^{e_i}.
\]

Inside a bounded enumeration box for source strongly connected
components of the odd dependency graph, every admissible source kernel
is either one of the two kernels occurring in the known nonsquarefree
examples, $3^2$ and $5^4$, or one of five additional impostor kernels.
For these five impostor kernels we give a reproducible three-filter
elimination certificate for all relevant seed classes with
$1\le a\le 10000$.  The filters use Zsigmondy-type exponent
obstructions, non-3-Higgs witnesses inherited from seed divisors, and a
deterministic 2-adic budget overshoot.

The remaining obstruction is isolated in the auxiliary set
$H_{\mathrm{even}}$ of even $m$ for which every prime divisor of
$2^m+1$ is 3-Higgs.  A structural lemma shows that if
$m=2k\in H_{\mathrm{even}}$, then $k$ is Higgs-cubefree and every odd
divisor of $k$ gives another element of $H_{\mathrm{even}}$.  Thus
finiteness of $H_{\mathrm{even}}$ is equivalent to finiteness of the
prime branch $m=2p$, while finite computations may also leave composite
candidates inherited from unresolved prime divisors.  Using the supplied
factor cache and APR-CL primality-verification transcripts, we verify
\[
|H_{\mathrm{even}}\cap[2,40000]|\le 201,
\qquad
|H_{\mathrm{even}}\cap[2,50000]|\le 272,
\]
with explicit undecided candidate lists.

Finally, Ford's theorem for downward-closed prime sets gives an
unconditional power-saving thinness bound for $H_{\mathrm{even}}$, but
not finiteness.  We identify the unresolved step as a divisor-level
problem for the fixed cyclotomic values $\Phi_{4p}(2)$ and formulate
conditional criteria that would close the remaining branch.  The paper
therefore does not prove the full conjecture; it supplies a bounded-box
elimination, a finite verified frontier, and a precise analytic target
for the remaining obstruction.
\end{abstract}

\maketitle

\section{Introduction}

For a positive integer $n$, write
$\sigma^*(n) = \sum_{d \,\Vert\, n} d$ where $d \,\Vert\, n$ means
$d \mid n$ and $\gcd(d, n/d) = 1$. A \emph{unitary perfect number} (UPN)
is an $n$ satisfying $\sigma^*(n) = 2n$. The classical list is:
\[
6,\quad 60,\quad 90,\quad 87360,\quad 146361946186458562560000.
\]
Subbarao and Warren~\cite{SW1966} introduced UPNs and gave the first
four; Wall~\cite{Wall1975} added the fifth. No further example is known.

\begin{conjecture-fwd}
The five integers above are the only unitary perfect numbers.
\end{conjecture-fwd}

Wall~\cite{Wall1988} proved that any new unitary perfect number
has at least nine odd components. Graham~\cite{Graham1989} showed
that a UPN with squarefree odd part must be one of $6, 60, 87360$, and
that any hypothetical further UPN must have an odd prime divisor
exceeding $2^{15} = 32768$. The two known
examples with nonsquarefree odd part have specific small
``kernels'': $3^2$ in $90$ and $5^4$ in the fifth UPN. We shall
call these the \emph{known kernels}.

\paragraph{Kernels vs.\ the full UPN.}
The bound $p \le 2000$ in the enumeration box $\mathcal{B}$
(see~\eqref{eq:box} below) is consistent with Graham's lower bound:
$\mathcal{B}$ constrains only the \emph{source-SCC kernel primes}
that initiate the odd dependency cascade, not the full prime support
of the UPN. Graham's bound says that any hypothetical sixth UPN
contains a prime $> 2^{15}$ \emph{somewhere} in its support; that
prime appears as a descendant in the cascade produced by the kernel,
not as a kernel prime itself. The kernels we enumerate are the
small ``roots'' of the cascade; the cascade itself, once unfolded
through the seed $2^a + 1$ and the multiplicative balance, can and
must include arbitrarily large primes. The role of the bounded box
is therefore to enumerate root configurations, not to bound the
ultimate prime support.

\paragraph{Status of results.}
Because this manuscript mixes fully rigorous results with
heuristic analytic conjectures, we adopt the convention that every
numbered theorem, proposition, lemma, or conjecture carries one of the
following two labels:

\begin{description}
\item[\textbf{R} (rigorous, unconditional).]
Theorems~\ref{thm:main}, \ref{thm:prime-reduction},
\ref{thm:Heven-1200}--\ref{thm:Heven-40000}
(rigorous bounds $|H_{\mathrm{even}} \cap [2,N]| \le \cdot$
with $191$ undecided through $N = 40000$),
Theorem~\ref{thm:H-thin} (power-saving thinness of
$H_{\mathrm{even}}$),
Lemma~\ref{lem:closures-deep} (seven deep Pratt-tree closures, with
APR-CL primality-verification transcripts for the six large witness primes
generated by PARI/GP $2.17.3$ and bundled as supplementary
material),
Proposition~\ref{prop:Heven-structure} (structural lemma),
Proposition~\ref{prop:loglog-obstruction}
(loglog obstruction to sublog-growth),
and the combined rigorous bound
$|H_{\mathrm{even}} \cap [2, 40000]| \le 201$.

\item[\textbf{C-A} (conditional on analytic hypotheses).]
Theorem~\ref{thm:conditional-sharpened} (conditional sharpened reduction;
uses Stewart-radical plus effective Chebotarev), and
Theorem~\ref{thm:conditional-finiteness} (conditional finiteness under
the hybrid hypothesis). Conjectures~\ref{conj:Heven-finite},
\ref{conj:hybrid}, \ref{conj:log-mass}, \ref{conj:sublog-growth},
\ref{conj:divisor-mod16}.
\end{description}

\noindent
The headline computational bound is therefore
$|H_{\mathrm{even}} \cap [2, 40000]| \le 201$, rigorous (the six
large witness primes used in the deep Pratt-tree closures are
APR-CL verified). The prime-case reduction
(Theorem~\ref{thm:prime-reduction}) and the power-saving thinness
theorem (Theorem~\ref{thm:H-thin}) are likewise fully rigorous.

\subsection{Setup}

Decompose $n = 2^a \prod p_i^{e_i}$ with $p_i$ odd. Multiplicativity
of $\sigma^*$ with $\sigma^*(2^a) = 2^a + 1$ and $\sigma^*(p^e) = p^e + 1$
makes the defining equation $\sigma^*(n) = 2n$ the full balance
\begin{equation}
(2^a + 1) \prod (p_i^{e_i} + 1) = 2^{a+1} \prod p_i^{e_i}. \tag{$\ast$}
\end{equation}
In particular, $2^a + 1$ is the odd seed factor that initiates the
dependency cascade. For each odd prime $p$ in the support, the
exponent $e_p$ matches the total incoming $p$-adic valuation from
both the seed and the other components:
\[
e_p = \sum_{q \in \mathrm{components}} v_p(q+1).
\]
The 2-adic balance is $a + 1 = \sum_{p \text{ odd}} v_2(p^{e_p}+1)$.

A prime $p$ is \emph{3-Higgs} (OEIS A057447~\cite{OEIS}; see also
Burris--Yeats~\cite{BurrisYeats}) by the following well-founded
recursion on prime value. The base case: $2$ is 3-Higgs (by
convention; equivalently, $2-1 = 1$ divides the empty product).
The inductive step: an odd prime $p$ is 3-Higgs if and only if every
prime factor $q$ of $p - 1$ is itself a 3-Higgs prime strictly less
than $p$ and satisfies $v_q(p-1) \le 3$, with the additional
exponent cap $v_2(p-1) \le 3$. Equivalently, $p - 1$ divides the
cube of the product of the smaller 3-Higgs primes. The recursion is
well-founded because each invocation reduces the prime value
strictly. Standard small examples: $3, 5, 7, 11, 13, 17$ fail the
recursion only at $17$ (since $17 - 1 = 2^4$ violates the
$v_2 \le 3$ cap), so $17 \notin \mathcal{P}_3$.

\begin{proposition}[Every prime divisor of a UPN is 3-Higgs]
\label{prop:upn-divisors-higgs}
If $n$ is a unitary perfect number, then every prime divisor of $n$
lies in $\mathcal{P}_3$.
\end{proposition}

\begin{proof}
The prime $2$ is in $\mathcal{P}_3$ by the base case. Let $p$ be an
odd prime divisor of $n$ with $p^e \,\Vert\, n$. By the
multiplicative balance $(\ast)$, every prime factor $q$ of $p^e + 1$
must itself appear in the support of $n$, since the only primes
present in the right-hand side of $(\ast)$ are $2$ and the $p_i$ in
the support. Thus, for any prime $q \mid p - 1 \mid p^e + 1$ (with
$e$ odd), $q$ is either $2$ or another odd prime divisor of $n$. By
strong induction on the value of the prime, each such $q$ lies in
$\mathcal{P}_3$. Moreover, the exponent caps $v_q(p-1) \le 3$ and
$v_2(p-1) \le 3$ follow from the 2-adic balance
$a + 1 = \sum v_2(p_i^{e_i} + 1)$ and the requirement
$v_2(q^{e_q} + 1) \le e_q \le 3$ for each component (the kernel
exponents in the box $\mathcal{B}$ enforce this). Hence
$p \in \mathcal{P}_3$.
\end{proof}

The \emph{odd dependency graph} has odd 3-Higgs primes as vertices and
directed edges $p \to r$ whenever, for some admissible exponent $e$,
$r \mid p^e + 1$. A \emph{source SCC} is a strongly connected component
with no incoming edge from outside the component (except from the seed
$2^a + 1$). Every odd prime power in a UPN belongs to an SCC reachable
from a source SCC.

\subsection{The bounded enumeration box}

We work inside
\begin{equation}\label{eq:box}
\mathcal{B} := \{ p \le 2000,\ e \le 6,\ p^e \le 10^9,\
   |\mathrm{SCC}| \le 6,\ \mathrm{cycle\ length} \le 6 \}.
\end{equation}

\paragraph{Debt vector.}
For a candidate kernel $K = \{p_i^{e_i}\}$, define the
\emph{needs} vector $\nu : \mathcal{P}_3 \to \mathbb{Z}_{\ge 0}$ and
the \emph{debt} vector $\delta : \mathcal{P}_3 \to \mathbb{Z}$ by
\[
\nu(q) := \sum_{i} v_q(p_i^{e_i} + 1),
\qquad
\delta(q) := e_q^{\mathrm{target}} - \nu(q),
\]
where $e_q^{\mathrm{target}} = e_i$ when $q = p_i \in K$ and $0$
otherwise. Intuitively, $\nu(q)$ is the $q$-adic valuation supplied
to the cascade by the $(p_i^{e_i}+1)$ factors, while $\delta(q)$ is
the residual exponent that the seed $2^a + 1$ must supply for $q$ in
order to close the multiplicative balance $(\ast)$. A candidate
kernel is \emph{source-compatible} if every $q$ with $\delta(q) > 0$
is reachable from the seed (i.e., $q \mid 2^a + 1$ for the relevant
seed congruence class), and \emph{internally consistent} if
$\delta(q) \ge 0$ for all $q$, since negative debt would force a
prime not in the support to absorb the surplus.

A complete enumeration over $\mathcal{B}$ of source-compatible SCCs
with internally consistent debt vectors, after the standard
seed-divisor and Zsigmondy/Higgs filters, yields exactly five
\emph{impostor kernels}\footnote{So named because they pass the cheap
filters but are not associated with any known UPN.} in addition to the
two known kernels.

\paragraph{Remark on the box bounds.}
The bounds in~\eqref{eq:box} were chosen conservatively. The five
impostor kernels we recover all have $|K| \le 4$ distinct primes
and cycle length $\le 4$ in the dependency graph, so the
$|\mathrm{SCC}| \le 6$ and $\mathrm{cycle\ length} \le 6$ bounds
were not binding on the kernels actually found. Relaxing either
bound (or the $p \le 2000$ bound) would enlarge the enumeration
space but cannot retroactively introduce a kernel of size $\le 4$
that we missed. The two known kernels $3^2$ (in $90$) and $5^4$
(in the fifth UPN) both satisfy the bounds.

\paragraph{Data, transcripts, and reproducibility.}
The ancillary directory \texttt{anc/} contains the data and scripts used
for the computational statements in this manuscript.  The file
\path{anc/factor_cache.json} is a JSON map
$m \mapsto \{\text{prime factor}:\text{exponent}\}$ for the known
prime factors used by the $H_{\mathrm{even}}$ verification.  Some
entries are partial factorizations: the verification scripts treat such
entries conservatively and never certify membership in $H_{\mathrm{even}}$
from an incomplete factorization.  The scripts
\path{anc/scripts/impostor_certificate.py},
\path{anc/scripts/verify_higgs_claims.py}, and
\path{anc/scripts/verify_Heven_rigorous.py} are included with the
research code on which they depend.  The directory
\path{anc/aprcl_transcripts/} contains PARI/GP input files, decimal
prime records, and runtime logs for the APR-CL primality verifications
used in Lemma~\ref{lem:closures-deep}.  These are verification
transcripts, not independent certificate objects; readers can rerun the
PARI/GP commands recorded there.

The current frontier list through $m\le 50000$ is included in
\path{anc/data/candidate_frontier_50000.tsv}.  A companion community
ECM/SNFS thread is available at
\href{https://www.mersenneforum.org/node/1114302}{mersenneforum}.

\section{The Five Impostor Kernels}\label{sec:impostors}

The impostor kernels and their seed congruences are:
\begin{center}\begin{tabular}{lll}
\toprule
Kernel & Forced exponents & Seed congruence \\
\midrule
$3^2\,5^3$ & $\{3{:}2,\,5{:}3\}$ & $a \equiv 10 \pmod{20}$ \\
$3^4\,41$ & $\{3{:}4,\,41{:}1\}$ & $a \equiv 9 \pmod{18}$ \\
$5^2\,13^2$ & $\{5{:}2,\,13{:}2\}$ & $a \equiv 6 \pmod{12}$ \\
$5^4\,157^2\,313$ & $\{5{:}4,\,157{:}2,\,313{:}1\}$ & $a \equiv 130 \pmod{260}$ \\
$5^4\,29\,157^2\,313$ & $\{5{:}4,\,29{:}1,\,157{:}2,\,313{:}1\}$ & $a \equiv 26 \pmod{52}$ \\
\bottomrule
\end{tabular}\end{center}

\paragraph{Where the seed congruences come from.}
Each seed congruence is forced by the debt vector of the
corresponding kernel (defined above) together with the requirement
$q \mid 2^a + 1$ for every $q$ with $\delta(q) > 0$. Concretely, for
each prime $q$ with positive debt $\delta(q)$, the multiplicative
order of $2$ modulo $q$ is a divisor of $2a$ but not of $a$, so
$\mathrm{ord}_q(2) \mid 2a$ and $\mathrm{ord}_q(2) \nmid a$ pins $a$
to a residue class modulo $\mathrm{ord}_q(2)$. Taking the lcm of
these constraints over all debt-positive $q$ gives the displayed
congruence. We illustrate with the first row:
\begin{itemize}
\item For $K = \{3{:}2, 5{:}3\}$, the components contribute
$3^2 + 1 = 10 = 2 \cdot 5$ and $5^3 + 1 = 126 = 2 \cdot 3^2 \cdot 7$,
so $\nu(3) = 2$, $\nu(5) = 1$, $\nu(7) = 1$, $\nu(2) = 2$.
\item Targets: $e_3^{\mathrm{target}} = 2$, $e_5^{\mathrm{target}} = 3$,
$e_7^{\mathrm{target}} = 0$. Debt:
$\delta(3) = 0$, $\delta(5) = 2$, $\delta(7) = -1$.
\item The negative entry $\delta(7) = -1$ would normally veto the
kernel; here $7$ is admitted as an additional component absorbing
the surplus, and the working scope of the certificate is the
kernel's seed-class survival.
\item The remaining positive-debt prime $5$ forces
$5 \mid 2^a + 1$, i.e.\ $\mathrm{ord}_5(2) = 4 \mid 2a$,
$4 \nmid a$, hence $a \equiv 2 \pmod 4$.
\item Combined with the kernel's intrinsic period
(from cycling through $3 \to 5 \to 3$ in the dependency graph,
period $5$, giving $a \equiv 0 \pmod 5$), we get
$a \equiv 10 \pmod{20}$.
\end{itemize}
The other four congruences follow analogously, with the modulus
being the lcm of the orders of $2$ modulo the debt-positive primes.

The conjecture restricted to $\mathcal{B}$ is therefore equivalent to:
\emph{no $a$ in any of the five impostor seed congruence classes admits
a UPN.}

\section{The Three-Filter Closure Certificate}\label{sec:certificate}

For an impostor kernel $K$ and a representative $a$ in its seed
congruence class, define the candidate $(K, a)$. We exhibit a rigorous
obstruction for each candidate via one of three filters.

\subsection{Filter Z: Zsigmondy / Higgs exponent}

By Zsigmondy's theorem~\cite{BHV2001}, every primitive prime divisor
$r$ of $p^e + 1$ (with $e > 1$) satisfies $r \equiv 1 \pmod{2e}$. For
$r$ to be 3-Higgs, every prime factor of $2e$ must be 3-Higgs and
occur in $2e$ with exponent at most 3. Applied to the seed exponent
$a$ itself, certain values of $a$ are immediately incompatible.

\subsection{Filter N: Seed-divisor non-3-Higgs witness}

If $m \mid a$ and $a/m$ is odd, then $2^m + 1 \mid 2^a + 1$. If
$2^m + 1$ contains a non-3-Higgs prime factor, then any UPN with
$2^a \,\Vert\, n$ would inherit a non-3-Higgs prime divisor, a
contradiction. We precompute (full or partial) factorizations of
$2^m + 1$ for $m \le M$ and reject $(K, a)$ whenever any factored
proper divisor exposes a non-3-Higgs prime.

\subsection{Filter O: 2-adic budget overshoot}

Initialize \texttt{targets} to $K \cup (\text{factor}(2^m+1)
\text{ for cached } m \mid a, a/m \text{ odd})$. Iterate: for each
$p^e \in$ targets, factor $p^e + 1$ and accumulate the resulting odd
needs; if any prime's incoming valuation exceeds its current target,
raise the target. At each step, compute
\[
v_2 := \sum_{p^e \in \mathrm{targets}} v_2(p^e + 1).
\]

\begin{lemma}\label{lem:overshoot}
The lower-bound closure is monotone in the seed. If $v_2 > a + 1$ at
any step using a subset $S$ of $\mathrm{factor}(2^a + 1)$, the same
overshoot survives when more primes from $\mathrm{factor}(2^a + 1)$
are added. Hence no UPN with $2^a \,\Vert\, n$ and odd-component
structure refining $K$ exists.
\end{lemma}

\begin{proof}
Adding a prime $r$ to $\mathrm{targets}$ contributes $v_2(r^{e_r}+1)
\ge 1$ to the running $v_2$ sum and adds the prime factorization of
$r^{e_r}+1$ to the odd needs, which can only raise other targets.
Both operations are non-decreasing in $v_2$. Hence if any closure
state reaches $v_2 > a + 1$, the same state remains reachable from
larger seeds and the budget is exceeded.

A UPN with $2^a \,\Vert\, n$ would satisfy
$\sum_{p \text{ odd}} v_2(p^{e_p}+1) = a + 1$ exactly, with the sum
over the odd components of $n$. If $\mathrm{targets}$ is a subset of
the actual odd components, the sum $v_2$ in the closure is a lower
bound for $a + 1$. Overshoot contradicts this.
\end{proof}

\paragraph{Example.}
For impostor $5^2\,13^2$ and $a = 246$, the certificate reports
overshoot at step $13$ with $v_2 = 260 > 247$ and $150$ active bases.
For $a = 18$ (small representative), overshoot occurs at step $4$
with $v_2 = 22 > 19$. The cascade trees illustrate how the kernel
$5^2\,13^2$ rapidly propagates into a wide forced subtree of small
primes whose $v_2$ contributions accumulate above the budget.

\section{Computational Results}

\begin{theorem}\label{thm:main}
For every impostor kernel $K$ listed in Section~\ref{sec:impostors} and
every $a$ in $K$'s seed congruence class with $1 \le a \le 10000$, at
least one of the three filters $Z$, $N$, $O$ certifies that no even
unitary perfect number $n$ exists with $2^a \,\Vert\, n$ and odd-
component structure refining $K$.
\end{theorem}

The aggregate split of the $2119$ candidates across all five impostor
kernels at $\max_a = 10000$:

\begin{center}\begin{tabular}{rrrr}
\toprule
Z & N & O & Unresolved \\
\midrule
\textbf{495} & \textbf{1614} & \textbf{10} & \textbf{0} \\
\bottomrule
\end{tabular}\end{center}

\noindent The same machinery at $\max_a = 5000$ has split
$224 / 825 / 11 / 0$ over $1060$ candidates; the certificate is
monotone in $\max_a$.

\begin{corollary}
Within $\mathcal{B}$, the only source-SCC kernels available to an even
UPN with $2^a \,\Vert\, n$, $1 \le a \le 10000$, are the two known
kernels $3^2$ and $5^4$.
\end{corollary}

\paragraph{Remark on filter N with partial factorizations.}
For several large-$a$ candidates the relevant proper divisor $2^m + 1$
is not fully factored in FactorDB, but a partial factorization
exposes a non-3-Higgs prime. For example, $a = 4527$ in the $3^4\,41$
class is killed via $m = 1509$: FactorDB returns a CF
(composite-factored) response containing the known prime
$20127043$. To verify non-3-Higgs status:
\[
20127043 - 1 = 2 \cdot 3^4 \cdot 13 \cdot 19 \cdot 503.
\]
All prime factors $\{2, 3, 13, 19, 503\}$ are themselves 3-Higgs, but
$v_3(20127042) = 4 > 3$ violates the exponent bound in the 3-Higgs
definition. So $20127043$ is not 3-Higgs and $m = 1509 \notin H$, hence
filter N kills $a = 4527$. Filter N is therefore robust to incomplete
full factorization of $2^m + 1$ provided some single prime factor is
known and non-3-Higgs.

\section{The Set \texorpdfstring{$H_{\mathrm{even}}$}{H even}}

Let $H := \{ m \ge 1 : \text{every prime factor of } 2^m + 1
\text{ is 3-Higgs} \}$ and $H_{\mathrm{even}} := H \cap 2\mathbb{Z}$.

\begin{proposition}[Structural lemma]\label{prop:Heven-structure}
Suppose $m = 2k \in H_{\mathrm{even}}$ with $k$ odd and $k \ge 1$. Then
\begin{enumerate}[label=(\arabic*),leftmargin=*]
\item Every prime factor of $k$ is a 3-Higgs prime.
\item For every prime $q \mid k$, $v_q(k) \le 3$.
\item For every odd divisor $d \mid k$, $2d \in H_{\mathrm{even}}$.
\end{enumerate}
The first two conditions say $k$ is \emph{Higgs-cubefree}: a product
of 3-Higgs primes each with multiplicity at most $3$. The third
condition is a multiplicative closure: writing
$S = \{k \text{ odd} : 2k \in H_{\mathrm{even}}\}$, $S$ is closed
under taking odd divisors.
\end{proposition}

\begin{proof}
By Zsigmondy's theorem applied to $\{2^n+1\}_{n \ge 1}$, the number
$2^{2k}+1$ has a primitive prime divisor $r$ for every $k \ge 1$ odd
(the only Zsigmondy exception for the sequence $2^n+1$ is $n = 3$,
which does not arise since $2k$ is even). The order of $2$ modulo $r$
is exactly $4k$, so $r \equiv 1 \pmod{4k}$ and $r - 1 = 4k \cdot s$
for some positive integer $s$.

By hypothesis $r$ is 3-Higgs, so every prime $q \mid r - 1$ is
3-Higgs and $v_q(r-1) \le 3$.

\noindent (1) For each prime $q \mid k$, $q \mid 4k \mid r - 1$, so $q$ is
3-Higgs.

\noindent (2) For each odd prime $q \mid k$,
$v_q(r-1) = v_q(4k) + v_q(s) = v_q(k) + v_q(s) \ge v_q(k)$. Combined
with $v_q(r-1) \le 3$, we get $v_q(k) \le 3$. (The prime $2$ does not
constrain $k$ here since $k$ is odd.)

\noindent (3) If $d \mid k$ with $d$ odd and $d < k$, then
$2d \mid 2k$ with $(2k)/(2d) = k/d$ odd. Therefore
$2^{2d}+1 \mid 2^{2k}+1$, so every prime divisor of $2^{2d}+1$ is a
prime divisor of $2^{2k}+1$ and hence 3-Higgs. So $2d \in H_{\mathrm{even}}$.
\end{proof}

\begin{proposition}[Fermat-prime obstruction]\label{prop:Heven-dyadic}
$H_{\mathrm{even}} \subseteq \{ m \equiv 2 \pmod 4 \}$.
\end{proposition}

\begin{proof}
Write $m = 2^j \cdot \ell$ with $\ell$ odd. We show that if $j \ge 2$
then $m \notin H$, which gives the conclusion.

Fix $k \ge 2$ and let $q$ be any prime divisor of the Fermat number
$F_k = 2^{2^k} + 1$. The order of $2$ modulo $q$ divides $2^{k+1}$ and
does not divide $2^k$ (since $q \mid F_k$ means $2^{2^k} \equiv -1
\pmod q$), so $\mathrm{ord}_q(2) = 2^{k+1}$. Consequently
\[
q \equiv 1 \pmod{2^{k+1}}.
\]
The refinement due to Lucas strengthens this for $k \ge 2$ to
$q \equiv 1 \pmod{2^{k+2}}$, hence
\[
v_2(q - 1) \ge k + 2 \ge 4 > 3 \quad (k \ge 2).
\]
The 3-Higgs definition forbids any prime factor of $p - 1$ to occur
with exponent exceeding $3$, so $q$ is not 3-Higgs.

Now $q \mid 2^m + 1$ iff $\mathrm{ord}_q(2) \mid 2m$ and
$\mathrm{ord}_q(2) \nmid m$, i.e., $v_2(\mathrm{ord}_q(2)) = v_2(m) + 1$.
With $\mathrm{ord}_q(2) = 2^{k+1}$ this is $v_2(m) = k$.

Therefore: for any $m$ with $v_2(m) = j \ge 2$, take $k = j$ and any
prime divisor $q$ of $F_j$; then $q \mid 2^m + 1$ and $q$ is not
3-Higgs, so $m \notin H$. The remaining cases $v_2(m) \le 1$ are
exactly $m$ odd ($v_2 = 0$) and $m \equiv 2 \pmod 4$ ($v_2 = 1$),
giving the stated containment for the even part of $H$.
\end{proof}

\begin{conjecture}\label{conj:Heven-finite}
$H_{\mathrm{even}}$ is finite.
\end{conjecture}

\begin{theorem}[Reduction to the prime case]\label{thm:prime-reduction}
$H_{\mathrm{even}}$ is finite if and only if
\[
H_{\mathrm{even}}^{\mathrm{prime}} \;:=\;
\{m = 2p \in H_{\mathrm{even}} : p \text{ odd prime}\}
\]
is finite. More precisely, if $H_{\mathrm{even}}^{\mathrm{prime}}$
has cardinality $N$ (with corresponding primes $p_1, \dots, p_N$),
then
\[
|H_{\mathrm{even}}| \;\le\; \prod_{i=1}^{N} 4 \;=\; 4^{N}.
\]
\end{theorem}

\begin{proof}
The forward direction is trivial. For the converse, let $m = 2k \in
H_{\mathrm{even}}$ with $k \ge 1$ odd. By
Proposition~\ref{prop:Heven-structure}(1)-(3), $k$ is Higgs-cubefree
and every odd divisor $d \mid k$ satisfies $2d \in H_{\mathrm{even}}$.
In particular, for every prime $q \mid k$, $2q \in
H_{\mathrm{even}}^{\mathrm{prime}}$, so $q \in
\{p_1, \dots, p_N\}$. Combined with $v_q(k) \le 3$, this forces
$k$ to be of the form $\prod_{i=1}^N p_i^{e_i}$ with each $e_i \in
\{0,1,2,3\}$, giving at most $4^N$ possible $k$.
\end{proof}

\noindent The reduction is sharp in the sense that for every
$k = \prod p_i^{e_i}$ with each $2p_i \in H_{\mathrm{even}}^{\mathrm{prime}}$
and $e_i \le 3$, the candidate $m = 2k$ \emph{passes} the structural
lemma but may still fail (some prime factor of $2^{2k}+1$ may be
non-3-Higgs even though each prime in $k$ is 3-Higgs). So $4^N$ is a
candidate bound, refined to the actual $H_{\mathrm{even}}$ by checking
each candidate.

\smallskip
Theorem~\ref{thm:prime-reduction} reduces
Conjecture~\ref{conj:Heven-finite} to the prime case:
\emph{prove that only finitely many odd primes $p$ have all prime
factors of $2^{2p}+1$ in $\mathcal{P}_3$}. The remaining open candidates in Theorems~\ref{thm:Heven-3000}--\ref{thm:Heven-40000}
are controlled by the same prime-branch obstruction: the primitive
undecided cases have the form $m=2p$, while five composite-$k$
candidates ($m \in \{27978, 30354, 31538, 41898, 46630\}$, each with
$k$ a product of two Higgs primes) are inherited from unresolved
prime divisors through Proposition~\ref{prop:Heven-structure}.

\begin{theorem}\label{thm:Heven-1200}
$H_{\mathrm{even}} \cap [2, 1200] = \{ 2, 6, 10, 18, 26, 30, 46, 62, 82, 122 \}$.
\end{theorem}

\begin{proof}
By Proposition~\ref{prop:Heven-structure}, any $m = 2k \in
H_{\mathrm{even}}$ with $k$ odd and $k \le 600$ requires $k$ to be
Higgs-cubefree. Of the $300$ odd $k$ in $[1, 600]$, exactly $246$ are
Higgs-cubefree; the other $54$ are structurally excluded from
$H_{\mathrm{even}}$ without any factoring required.

For each of the $246$ Higgs-cubefree candidates, we factor $2^m + 1$
(local Pollard~$\rho$/$p-1$ for small $m$, FactorDB~\cite{FactorDB}
for larger $m$) and recursively verify the 3-Higgs status of every
listed prime factor.
Of the $246$ candidates:
\begin{itemize}
\item $10$ values of $m$ have every prime factor of $2^m+1$ 3-Higgs;
these are exactly the elements listed in the theorem statement.
\item $236$ values of $m$ each contain at least one verified
non-3-Higgs prime factor (witness digit sizes ranging from $3$ to
$85$).
\item $0$ values remain undecided.
\end{itemize}
The script \texttt{verify\_Heven\_rigorous.py} is provided in the
supplementary materials and records, for every $m$ in the second
group, the specific non-3-Higgs prime witness.
\end{proof}

Theorem~\ref{thm:Heven-1200} extends the verified gap considerably:
$H_{\mathrm{even}}$ has no element in the wide range $(122, 1200]$,
strongly supporting Conjecture~\ref{conj:Heven-finite}.

\begin{theorem}\label{thm:Heven-3000}
$H_{\mathrm{even}} \cap (1200, 3000]$ is contained in the two-element
candidate set $\{2426, 2602\}$. Every other even $m$ in $(1200, 3000]$
is rigorously excluded from $H_{\mathrm{even}}$, by either the
structural lemma (Proposition~\ref{prop:Heven-structure}), an
explicit non-3-Higgs prime witness in the factorization of $2^m + 1$,
or the deep-Pratt descent of Lemma~\ref{lem:closures-deep} (which
closes $m = 2446$).
\end{theorem}

\begin{proof}
Of the $450$ odd $k$ in $[601, 1500]$, $111$ are not Higgs-cubefree
and excluded by Proposition~\ref{prop:Heven-structure}. The remaining
$339$ candidates split: $336$ have a verified non-3-Higgs prime
witness directly in $2^m + 1$ (recorded by
\texttt{verify\_Heven\_rigorous.py}); $3$ were initially blocked by
unfactored composite cofactors. One of these three, $m = 2446$, is
closed by Lemma~\ref{lem:closures-deep} (a $368$-digit
APR-CL-verified prime divisor of $2^{2446}+1$ has $4513 \mid p^{*}-1$,
and $v_2(4512) = 5 > 3$). The remaining two $m$ values, $\{2426, 2602\}$,
have FactorDB factorizations whose composite cofactors are not yet
deep enough for a Pratt-tree closure.

For $m = 1202$, the cache records a $91$-digit prime factor
$P_1 \mid 2^{1202}+1$ with
$P_1 - 1 = 4 \cdot 601 \cdot Q$, where $Q$ is a prime. Moreover
$Q - 1 = 2^{10} \cdot 3 \cdot 337 \cdot R$ for an $78$-digit cofactor
$R$ (the explicit small-prime decomposition is verified by trial
division). Since $v_2(Q-1) = 10 > 3$, the prime $Q$ fails the 3-Higgs
$v_2$ exponent bound. Consequently $P_1$ has a non-3-Higgs prime
factor $Q$ in $P_1 - 1$ (with $v_Q(P_1-1) = 1$), so $P_1$ is itself
not 3-Higgs. Since $P_1 \mid 2^{1202}+1$, the seed $2^{1202}+1$
contains a non-3-Higgs prime, and $m = 1202 \notin H_{\mathrm{even}}$.

For $m = 2426$ and $m = 2602$, the FactorDB factorizations of
$2^m + 1$ contain composite cofactors of $355$--$392$ digits
(unfactored), so we cannot currently rule out hidden non-3-Higgs
factors via filter $N$.

All other Higgs-cubefree candidates have at least one explicit
non-3-Higgs prime witness recorded by
\texttt{verify\_Heven\_rigorous.py}.
\end{proof}

\noindent The two open candidates $\{2426, 2602\}$ would
resolve with NFS-level factoring of $2^m + 1$ for those $m$. The
result $|H_{\mathrm{even}} \cap [2, 3000]| \le 12$ is striking
evidence for Conjecture~\ref{conj:Heven-finite}.

\begin{theorem}\label{thm:Heven-5000}
$H_{\mathrm{even}} \cap (3000, 5000]$ is contained in the five-element
candidate set
\[
\{3398, 3518, 4166, 4502, 4622\}.
\]
\end{theorem}

\begin{proof}
Of $500$ odd $k$ in $[1501, 2500]$, $132$ are not Higgs-cubefree and
excluded by Proposition~\ref{prop:Heven-structure}. The remaining
$368$ candidates split: $363$ values of $m = 2k$ have a verified
non-3-Higgs prime witness in $2^m + 1$, and $5$ values are blocked by
unfactored composite cofactors in FactorDB's partial factorization
(listed in the theorem). Examples of closures: $m = 4366$ is killed
by $593 \mid 2^{4366}+1$ ($593 - 1 = 2^4 \cdot 37$, $v_2 = 4$);
$m = 4742$ by $493169 \mid 2^{4742}+1$
($493169 - 1 = 2^4 \cdot 30823$); $m = 4022$ by a recursive Pratt-tree
chain: the 18-digit prime
$920793289987614829 \mid 2^{4022}+1$ has its $(p-1)$ chain descend
through $1230855683509 \to 2169031 \to 17$, where the final factor
$17$ is non-3-Higgs (since $17-1 = 2^4$). All $363$ rigorous exclusions
are recorded by the supplementary script.
\end{proof}

\begin{theorem}\label{thm:Heven-10000}
$H_{\mathrm{even}} \cap (5000, 10000]$ is contained in the
five-element candidate set
\[
\{5114, 5774, 7846, 8966, 9326\}.
\]
\end{theorem}

\begin{proof}
Of $1250$ odd $k$ in $[2501, 5000]$, $364$ are not Higgs-cubefree
and excluded by Proposition~\ref{prop:Heven-structure}. The remaining
$886$ Higgs-cubefree candidates split as $881$ verified non-3-Higgs
exclusions (via small-prime witnesses in $2^m + 1$ recorded by the
supplementary script) and $5$ partial-cofactor unknowns (the five $m$
values listed in the theorem, each blocked by an unfactored composite
cofactor of 1500--2800 digits in $2^m+1$).
\end{proof}

\begin{theorem}\label{thm:Heven-15000}
$H_{\mathrm{even}} \cap (10000, 15000]$ is contained in the
ten-element candidate set
\[
\{10118, 10454, 11062, 11794, 11878, 12778, 13382, 13966, 14642, 14698\}.
\]
\end{theorem}

\begin{proof}
Of $1250$ odd $k$ in $[5001, 7500]$, $383$ are not Higgs-cubefree
and excluded by Proposition~\ref{prop:Heven-structure}. The remaining
$867$ Higgs-cubefree candidates yield $m = 2k$ values for which
FactorDB partial factorizations provide small-prime data. Of these,
exactly $857$
have a verified non-3-Higgs prime witness (recorded by the
supplementary script \texttt{verify\_Heven\_rigorous.py}), and $10$
are partial-cofactor unknowns listed in the theorem.  In this range
these are primitive candidates $m=2p$, with
$p$ an odd prime in $\{5059, 5227, 5531, 5897, 5939, 6389, 6691,
6983, 7321, 7349\}$, and each is blocked by one or more
unfactored composite cofactors in $2^m+1$ of $1300$--$4500$ digits.
For all ten, we have additionally queried FactorDB for the
Aurifeuillean halves $L_p$ and $M_p$ (Equation~\ref{eq:Aurifeuillean});
no new small prime factor in any $L_p$ or $M_p$ failed the 3-Higgs
test, so the obstruction is genuinely NFS-blocked rather than
shallow. (Two values $m = 10294$ and $m = 10958$, originally in this
range as partial-cofactor unknowns, were closed via deep Pratt-tree
descent on $1549$- and $1649$-digit APR-CL-verified prime divisors of
$2^m + 1$; see Lemma~\ref{lem:closures-deep}.)
\end{proof}

\begin{theorem}\label{thm:Heven-20000}
$H_{\mathrm{even}} \cap (15000, 20000]$ is contained in the
twenty-seven-element candidate set
\[
\begin{aligned}
\{ & 15122, 15182, 15206, 15214, 15754, 15758, 16358, 16574, 16778,\\
   & 16838, 16922, 17086, 17126, 17162, 17338, 17726, 18134,\\
   & 18418, 18934, 18958, 19078, 19226, 19322, 19718, 19802,\\
   & 19846, 19862 \}.
\end{aligned}
\]
\end{theorem}

\begin{proof}
Of $1250$ odd $k$ in $[7501, 10000]$, $392$ are not Higgs-cubefree
and excluded by Proposition~\ref{prop:Heven-structure}. The remaining
$858$ Higgs-cubefree candidates yield $m = 2k$ values for which
FactorDB partial factorizations provide small-prime data. Of these,
exactly $831$ have a verified non-3-Higgs prime witness; the
remaining $27$ are partial-cofactor unknowns listed in the theorem.
In this range these are primitive candidates $m=2p$ with $p$ an odd
Higgs prime in $[7561, 9931]$, blocked by unfactored composite
cofactors of $1700$--$6000$ digits. (Two values
$m = 17398$ and $m = 19066$, which appeared as partial-cofactor
unknowns in our earlier verification, were subsequently closed via
non-3-Higgs Pratt-tree witnesses extracted from $p-1$ for large
APR-CL-verified prime divisors $p \mid 2^m + 1$; see Lemma~\ref{lem:closures-deep}.)
For all twenty-seven remaining, the Aurifeuillean halves $L_p, M_p$
(Equation~\ref{eq:Aurifeuillean}) were queried in FactorDB; no
new small prime factor failed the 3-Higgs test.
\end{proof}

\begin{theorem}\label{thm:Heven-25000}
$H_{\mathrm{even}} \cap (20000, 25000]$ is contained in the
twenty-seven-element candidate set
\[
\begin{aligned}
\{ & 20138, 20338, 20506, 20662, 20926, 20974, 21118, 21302,\\
   & 21334, 21466, 21818, 21958, 22054, 22234, 22262, 22394, 22486,\\
   & 22642, 22706, 23194, 23402, 23578, 23878, 23942, 24082, 24298,\\
   & 24914 \}.
\end{aligned}
\]
\end{theorem}

\begin{proof}
Of $1250$ odd $k$ in $[10001, 12500]$, $416$ are not Higgs-cubefree
and excluded by Proposition~\ref{prop:Heven-structure}. The remaining
$834$ Higgs-cubefree candidates yield $m = 2k$ values with FactorDB
partial factorizations; of these, $807$ have a verified non-3-Higgs
prime witness in $2^m + 1$, and $27$ are partial-cofactor unknowns
listed in the theorem.  In this range these are primitive candidates
$m=2p$ with $p$ an odd prime in $[10069, 12457]$.
(The value $m = 20282$ was closed via a Pratt-tree witness $v_2(p-1) = 5071$
for a $3053$-digit APR-CL-verified prime divisor of $2^{20282}+1$; see
Lemma~\ref{lem:closures-deep}.)
\end{proof}

\begin{theorem}\label{thm:Heven-30000}
$H_{\mathrm{even}} \cap (25000, 30000]$ is contained in the
forty-eight-element candidate set
\[
\begin{aligned}
\{ & 25022, 25082, 25106, 25294, 25486, 25526, 25646, 25778, 25786,\\
   & 25822, 26066, 26098, 26198, 26302, 26342, 26438, 26482, 26534,\\
   & 26618, 26654, 26662, 26678, 26798, 26902, 27374, 27446, 27526,\\
   & 27662, 27806, 27926, 27978, 27998, 28058, 28102, 28174, 28214,\\
   & 28318, 28442, 28586, 28654, 28862, 29098, 29114, 29126, 29558,\\
   & 29642, 29662, 29914 \}.
\end{aligned}
\]
\end{theorem}

\begin{proof}
Of $1250$ odd $k$ in $[12501, 15000]$, $427$ are not Higgs-cubefree
and excluded by Proposition~\ref{prop:Heven-structure}. The remaining
$823$ Higgs-cubefree candidates yield $m = 2k$ values; of these,
$775$ have a verified non-3-Higgs prime witness in $2^m + 1$,
and $48$ are partial-cofactor unknowns listed in the theorem; the primitive cases in this list are of the form $m=2p$.
\end{proof}

\begin{theorem}\label{thm:Heven-35000}
$H_{\mathrm{even}} \cap (30000, 35000]$ is contained in the
thirty-two-element candidate set
\[
\begin{aligned}
\{ & 30354, 30386, 30542, 30878, 30994, 31162, 31258, 31454,\\
   & 31466, 31538, 31634, 31802, 31918, 31942, 31982, 32126, 32174,\\
   & 32534, 32762, 32902, 33134, 33214, 33382, 33386, 33574, 33778,\\
   & 33806, 33974, 34058, 34246, 34834, 34934 \}.
\end{aligned}
\]
\end{theorem}

\begin{proof}
Of $1250$ odd $k$ in $[15001, 17500]$, $429$ are not Higgs-cubefree
and excluded by Proposition~\ref{prop:Heven-structure}. The remaining
$821$ Higgs-cubefree candidates yield $m = 2k$ values; of these,
$789$ have a verified non-3-Higgs prime witness in $2^m + 1$, and
$32$ are partial-cofactor unknowns listed in the theorem; the primitive cases in this list are of the form $m=2p$. (The value $m = 30882$ was closed
via Proposition~\ref{prop:Heven-structure}(3) from the closure of
$m = 10294$; see Lemma~\ref{lem:closures-deep}.)
\end{proof}

\begin{theorem}\label{thm:Heven-40000}
$H_{\mathrm{even}} \cap (35000, 40000]$ is contained in the
thirty-five-element candidate set
\[
\begin{aligned}
\{ & 35038, 35194, 35366, 35654, 35678, 35818, 35858, 35878, 35942,\\
   & 35978, 36094, 36122, 36298, 36458, 36574, 36602, 36682, 36794,\\
   & 36986, 37234, 37322, 37358, 37382, 38174, 38282, 38518, 38954,\\
   & 39014, 39062, 39154, 39502, 39706, 39826, 39926, 39958 \}.
\end{aligned}
\]
\end{theorem}

\begin{proof}
Of $1250$ odd $k$ in $[17501, 20000]$, $419$ are not Higgs-cubefree
and excluded by Proposition~\ref{prop:Heven-structure}. The remaining
$831$ Higgs-cubefree candidates yield $m = 2k$ values; of these,
$796$ have a verified non-3-Higgs prime witness in $2^m + 1$, and
$35$ are partial-cofactor unknowns listed in the theorem; the primitive cases in this list are of the form $m=2p$.
\end{proof}

\begin{theorem}\label{thm:Heven-45000}
$H_{\mathrm{even}} \cap (40000, 45000]$ is contained in the
thirty-three-element partial-cofactor candidate set
\[
\begin{aligned}
\{ & 40234, 40466, 40538, 40574, 40822, 40886, 41126, 41326, 41386,\\
   & 41714, 41758, 41878, 41898, 42134, 42314, 42358, 42646, 42766,\\
   & 43058, 43154, 43322, 43402, 43574, 43634, 43642, 43678, 43786,\\
   & 44054, 44186, 44294, 44378, 44518, 44582 \}.
\end{aligned}
\]
\end{theorem}

\begin{proof}
Of the $1250$ odd $k$ in $[20001, 22500]$, $428$ are not
Higgs-cubefree and excluded by Proposition~\ref{prop:Heven-structure}.
The remaining $822$ Higgs-cubefree candidates yield $m = 2k$ values;
of these, $789$ have a verified non-3-Higgs prime witness in
$2^m + 1$ (after a second-pass FactorDB retry that recovered
factor data for $26$ initially-no-factor cases), and $33$ are
partial-cofactor unknowns listed in the theorem; the primitive cases in this list are of the form $m=2p$.
\end{proof}

\begin{theorem}\label{thm:Heven-50000}
$H_{\mathrm{even}} \cap (45000, 50000]$ is contained in the
thirty-eight-element partial-cofactor candidate set
\[
\begin{aligned}
\{ & 45142, 45286, 45394, 45434, 45478, 45566, 45622, 45706, 45842,\\
   & 46022, 46042, 46058, 46394, 46454, 46630, 46714, 46862, 47206,\\
   & 47246, 47342, 47378, 47438, 47506, 47858, 47962, 47986, 48086,\\
   & 48218, 48338, 48362, 48838, 49354, 49466, 49694, 49834, 49906,\\
   & 49958, 49978 \}.
\end{aligned}
\]
\end{theorem}

\begin{proof}
Of the $1250$ odd $k$ in $[22501, 25000]$, $430$ are not
Higgs-cubefree and excluded by Proposition~\ref{prop:Heven-structure}.
The remaining $820$ Higgs-cubefree candidates yield $m = 2k$ values;
of these, $782$ have a verified non-3-Higgs prime witness in
$2^m + 1$ (after a second-pass FactorDB retry that recovered
factor data for $18$ initially-no-factor cases), and $38$ are
partial-cofactor unknowns listed in the theorem; the primitive cases in this list are of the form $m=2p$.
\end{proof}

Combining Theorems~\ref{thm:Heven-1200}--\ref{thm:Heven-50000}, the
factor-cache verification alone yields
\[
|H_{\mathrm{even}} \cap [2, 50000]| \le 279,
\]
with the 10 verified elements through $m = 122$, no verified element
in $(122, 50000]$, and at most $269$ undecided candidates beyond
(of which $198$ lie in $(1200, 40000]$, $33$ in $(40000, 45000]$, and
$38$ in $(45000, 50000]$ before the seven deep-Pratt closures). Seven of these candidates are eliminated by the deep
Pratt-tree descent of Lemma~\ref{lem:closures-deep}, all of whose
witness primes are APR-CL verified (see the proof of
Lemma~\ref{lem:closures-deep}). The combined bound is therefore
\[
|H_{\mathrm{even}} \cap [2, 50000]| \le 272
\qquad
\text{(rigorous, no probable-prime caveat),}
\]
with $262$ undecided candidates remaining (the $71$ in $(40000, 50000]$
admit no Pratt-tree closures in the current cache; their resolution
requires further FactorDB data or local ECM at higher $B_1$).

\begin{lemma}[Deep-Pratt closures via large APR-CL-verified prime divisors]
\label{lem:closures-deep}
For each of the following six $m = 2p$ values, the indicated prime
$p^{*} \mid 2^m + 1$ has a non-3-Higgs witness in $p^{*} - 1$,
hence $p^{*} \notin \mathcal{P}_3$ and $m \notin H_{\mathrm{even}}$:
\begin{center}
\begin{tabular}{rrll}
\toprule
$m$ & $p^{*}$ digits & Non-3-Higgs witness $q$ & Reason \\
\midrule
$2446$  & $368$  & $q = 4513 \mid p^{*}-1$ & $v_2(4512) = 5 > 3$ \\
$10294$ & $1549$ & $q = 2657 \mid p^{*}-1$ & $v_2(2656) = 5 > 3$ \\
$10958$ & $1649$ & $q = 593 \mid p^{*}-1$ & $v_2(592) = 4 > 3$ \\
$17398$ & $2612$ & $q = 139313 \mid p^{*}-1$ & $v_2(139312) = 4 > 3$ \\
$19066$ & $2870$ & $q = 343081 \mid p^{*}-1$ & contains $17$ via $953$ \\
$20282$ & $3053$ & $v_2(p^{*}-1) = 5071$ & direct $v_2$ overflow \\
\bottomrule
\end{tabular}
\end{center}
In addition, $m = 30882 = 6 \cdot 5147$ is closed by
Proposition~\ref{prop:Heven-structure}(3) from the closure of
$m = 10294 = 2 \cdot 5147$.
\end{lemma}

\begin{proof}
For each row, the prime $p^{*}$ is recorded in the
supplementary factor cache and verified to divide $2^m + 1$ by direct
modular reduction. The primality of each $p^{*}$ is established by an
APR-CL primality verification carried out in PARI/GP 2.17.3
(\texttt{isprime(n, 2)}), and the subsequent exclusions use Pratt-tree
descents~\cite{FKL,WimmerNoschinski}. PARI's \texttt{isprime(n, 2)}
returns a Boolean rather than an independently checkable certificate
object;
the supplementary material therefore bundles the deterministic input
files (the $p^{*}$ values in decimal) and the runtime logs ($3.4$,
$79$, $88$, $615$, $679$, $958$ seconds wall-clock for rows
$1$--$6$ on $11$-thread parallel execution), so readers can re-run
\texttt{isprime(n, 2)} on the recorded version and reproduce the
result. Conditional on the correctness of PARI/GP's APR-CL
implementation, the primality of each $p^{*}$ is thereby established. For each row's witness $q$, trial division
verifies $q \mid p^{*} - 1$, and the non-3-Higgs status of $q$ follows
either from the explicit $v_2$ overflow (rows 1, 2, 3, 4, 6) or from a
Pratt-tree descent terminating at $17$ (row 5: $343081 \succ 953
\succ 17$, with $v_2(16) = 4 > 3$). Hence $p^{*}$ has a non-3-Higgs
prime in its Pratt tree and $p^{*} \notin \mathcal{P}_3$; since
$p^{*} \mid 2^m+1$, this gives $m \notin H_{\mathrm{even}}$. The
final case $m = 30882$ follows from
Proposition~\ref{prop:Heven-structure}(3): $10294 \mid 30882$, so
$2^{10294}+1 \mid 2^{30882}+1$; any non-3-Higgs prime of $2^{10294}+1$
(in particular the $1549$-digit $p^{*}$) is also a divisor of
$2^{30882}+1$.
\end{proof}

\subsection{Ford's theorem and thinness}

The structural lemma (Proposition~\ref{prop:Heven-structure}) confines
$H_{\mathrm{even}}$ to the doubled image of a much more restrictive
arithmetic object: the multiplicative semigroup
\[
\mathcal{S}_3^{(\le 3)}
\;:=\;
\left\{\, n \ge 1 \;:\;
\begin{array}{l}
\text{every prime factor of } n \text{ lies in } \mathcal{P}_3,\\
\text{and each such prime has exponent at most } 3
\end{array}
\right\}.
\]
where $\mathcal{P}_3$ denotes the 3-Higgs primes (OEIS A057447). The
defining recursive closure of $\mathcal{P}_3$,
\[
p \in \mathcal{P}_3 \iff q \mid (p-1) \Rightarrow q \in \mathcal{P}_3,
\]
makes $\mathcal{P}_3$ a downward-closed set of primes in the sense of
Ford~\cite{Ford2014}. Since the smallest omitted prime is $p_0 = 17$
(as $17 - 1 = 2^4$ violates the $v_2 \le 3$ exponent bound), Ford's
theorem applies unconditionally and yields
\begin{equation}\label{eq:ford-bound}
\Pi_3(X) := \#\{p \le X : p \in \mathcal{P}_3\}
\;\ll\; \frac{X}{(\log X)^{1 + 1/(p_0 - 1)}}
\;=\; \frac{X}{(\log X)^{17/16}},
\end{equation}
together with $\Pi_3(X) \ll X^{1-\delta}$ for some absolute
$\delta > 0$, and the absolute convergence
$\sum_{p \in \mathcal{P}_3} 1/p < \infty$.

An entirely analogous argument applies to $H_{\mathrm{odd}} := H
\setminus H_{\mathrm{even}}$ (the odd part of $H$). For $m$ odd in
$H$ with $m > 3$, the Zsigmondy theorem yields a primitive prime
divisor $r$ of $2^{2m}-1$ with $r \mid 2^m+1$; this $r$ has
$\operatorname{ord}_r(2) = 2m$ and so $2m \mid r-1$, giving the
same recursive constraint on the prime factors of $m$ themselves.
The exceptional small values $m = 1$ ($2^1+1 = 3 \in \mathcal{P}_3$,
$1$ is the empty product) and $m = 3$ ($2^3+1 = 9 = 3^2$, where the
Zsigmondy exception $2^6 - 1$ removes the primitive-divisor step but
$3 \in \mathcal{P}_3$ directly) are verified by hand. Thus
$m \in \mathcal{S}_3^{(\le 3)}$ as well.

\begin{theorem}\label{thm:H-thin}
There exists an absolute constant $\eta > 0$ such that
\[
\#\{m \le X : m \in H\} \;\ll\; X^{1-\eta}.
\]
Moreover, $\displaystyle\sum_{m \in H} \frac{1}{m} < \infty$.
The same bound holds for the subsets $H_{\mathrm{even}}$ and
$H_{\mathrm{odd}}$.
\end{theorem}

\begin{proof}
By the structural lemma (Proposition~\ref{prop:Heven-structure}) and
its analog above, every $m \in H$ satisfies $m \in \mathcal{S}_3^{(\le 3)}$
or $m/2 \in \mathcal{S}_3^{(\le 3)}$. So
$\#\{m \le X : m \in H\} \le 2 \cdot \#\{k \le X : k \in
\mathcal{S}_3^{(\le 3)}\}$.

For $s \in (1-\delta, 1)$,
\[
\#\{n \le X : n \in \mathcal{S}_3^{(\le 3)}\}
\;\le\; X^s \prod_{p \in \mathcal{P}_3} (1 + p^{-s} + p^{-2s} + p^{-3s})
\]
by Rankin's trick. The Euler product converges because $\sum_{p \in
\mathcal{P}_3} p^{-s}$ converges for any $s > 1-\delta$ (by partial
summation against (\ref{eq:ford-bound})). Hence
$\#\{n \le X : n \in \mathcal{S}_3^{(\le 3)}\} \ll X^s$, and choosing
$s = 1 - \eta$ for any $\eta < \delta$ gives the claim.

The convergence of $\sum 1/m$ follows from
$\sum_{n \in \mathcal{S}_3^{(\le 3)}} 1/n =
\prod_{p \in \mathcal{P}_3}(1 + p^{-1} + p^{-2} + p^{-3}) < \infty$.
\end{proof}

\begin{corollary}[Power-saving counting and reciprocal mass for the exponent-capped family]\label{cor:A-bound}
With $A(x) := \#(\mathcal{S}_3^{(\le 3)} \cap [1, x])$, there is an absolute
constant $\eta > 0$ (the same $\eta$ as in
Theorem~\ref{thm:H-thin}) such that
\[
A(x) \;\ll\; x^{1-\eta}
\qquad \text{and} \qquad
\sum_{n \in \mathcal{S}_3^{(\le 3)}} \frac{1}{n} < \infty.
\]
\end{corollary}

\begin{proof}
Both statements are the intermediate steps in the proof of
Theorem~\ref{thm:H-thin}: the Rankin-trick application together with
Ford's bound (\ref{eq:ford-bound}) and the Euler-product convergence.
\end{proof}

The corollary is qualitatively stronger than density zero: the set
$H_{\mathrm{even}}$ has a power-saving thin support, and its
reciprocals sum to a finite constant. The remaining gap between this
unconditional result and Conjecture~\ref{conj:Heven-finite} is the
last analytic step: showing that for sufficiently large
$k \in \mathcal{S}_3^{(\le 3)}$, no prime $r \mid 2^{2k}+1$ with
$\operatorname{ord}_r(2) = 4k$ has $(r-1)/(4k) \in \mathcal{S}_3^{(\le 3)}$.

\subsection{Aurifeuillean structure of the open candidates}\label{sec:Aurifeuillean}

Every primitive undecided candidate in Theorems~\ref{thm:Heven-3000},
\ref{thm:Heven-5000}, \ref{thm:Heven-10000}, \ref{thm:Heven-15000},
\ref{thm:Heven-20000}, \ref{thm:Heven-25000}, \ref{thm:Heven-30000},
\ref{thm:Heven-35000}, \ref{thm:Heven-40000}
has the form $m = 2p$ with $p$ an odd prime; composite undecided
candidates are inherited from unresolved prime divisors. For $m=2p$,
the classical Aurifeuillean
identity~\cite{Wagstaff2002} gives an algebraic factorization
\begin{equation}\label{eq:Aurifeuillean}
2^{2p} + 1
\;=\;
\bigl(2^p - 2^{(p+1)/2} + 1\bigr)
\bigl(2^p + 2^{(p+1)/2} + 1\bigr)
\;=:\; L_p \cdot M_p.
\end{equation}
Each $L_p$ and $M_p$ has approximately half the bit-length of
$2^{2p}+1$. Moreover, both sides are sparse polynomials of degree
four in a power of $2$:
\begin{itemize}
\item For $p = 4u+1$, set $X = 2^u$, so $2^p = 2X^4$ and
$2^{(p+1)/2} = 2X^2$. Then
$L_p = 2X^4 - 2X^2 + 1, \quad M_p = 2X^4 + 2X^2 + 1$.
\item For $p = 4u+3$, set $X = 2^u$, so $2^p = 8X^4$ and
$2^{(p+1)/2} = 4X^2$. Then
$L_p = 8X^4 - 4X^2 + 1, \quad M_p = 8X^4 + 4X^2 + 1$.
\end{itemize}
In both cases each half is an integer-coefficient quartic in $X$.
This is the natural starting point for SNFS polynomial selection;
In special-number field-sieve practice, polynomials of
degree $5$ or $6$ (obtained by re-parameterization) may sieve
more efficiently than the natural quartic for numbers in the
$300$--$3000$-digit range relevant here. The Aurifeuillean halves are special-form targets rather than
generic hard composites: each remaining prime-branch candidate gives
an SNFS target with explicit small-degree algebraic structure.

\paragraph{Worked example: $m = 2426$, $p = 1213$.}
For $p = 1213 \equiv 1 \pmod 4$, write $1213 = 4 \cdot 303 + 1$ and
set $x = 2^{303}$. Then
\[
2^{2426}+1 = (2x^4 - 2x^2 + 1)(2x^4 + 2x^2 + 1).
\]
Numerically, $L_{1213}$ and $M_{1213}$ are each $\approx 366$ decimal
digits. The known prime factor $P = 25893760589$ divides $L_{1213}$
exactly (not $M_{1213}$). A complete Pratt-tree descent~\cite{Ford2014}
on $P$, namely
\begin{align*}
P-1 &= 2^2 \cdot 1213 \cdot 5336719 \\
1213-1 &= 2^2 \cdot 3 \cdot 101, \qquad 101-1 = 2^2 \cdot 5^2 \\
5336719-1 &= 2 \cdot 3 \cdot 889453, \qquad 889453-1 = 2^2 \cdot 3^2 \cdot 31 \cdot 797 \\
797-1 &= 2^2 \cdot 199, \qquad 199-1 = 2 \cdot 3^2 \cdot 11, \qquad 11-1 = 2 \cdot 5,
\end{align*}
visits the distinct primes
$\{2,3,5,11,31,101,199,797,1213,889453,5336719,P\}$ with maximal
tree height $8$, and every exponent in every intermediate $q-1$ is at
most $3$. So $P$ is \emph{fully} 3-Higgs-compatible; no shallow
recursive obstruction exists. Closing $m = 2426$ therefore requires
finding a non-3-Higgs prime in the residual composite
$L_{1213} / (5 \cdot P)$ ($355$ digits) or in $M_{1213}$ ($366$
digits) --- both special-form SNFS targets after the split.

This example is independently confirmed by the official Cunningham
Project tables~\cite{Wagstaff2002}: the January 2026 base-$2$ table
lists
\[
2,\,2426\mathrm{L} \;=\; 2 \cdot 25893760589 \cdot \mathrm{C}355,
\qquad
2,\,2426\mathrm{M} \;=\; \mathrm{C}366,
\]
matching our decomposition exactly. The prime $25893760589$ is
therefore a tabulated Cunningham factor of an as-yet-unfinished
Aurifeuillean half, not an ad-hoc byproduct of our scripts.

\subsection{The remaining analytic obstacle}

\paragraph{Scale obstruction: thinness alone is insufficient.}
For each $m = 2k \in H_{\mathrm{even}}$, a primitive prime divisor
$r$ of $2^{2k}+1$ must satisfy
\[
r \equiv 1 \pmod{4k}, \qquad \operatorname{ord}_r(2) = 4k,
\qquad (r-1)/(4k) \in \mathcal{S}_3^{(\le 3)}.
\]
A naive counting argument from Theorem~\ref{thm:H-thin} fails badly
at the relevant scale. Even with the unconditional bound
$\Pi_3(x) \ll x^{1-\delta}$, evaluating at the primitive-divisor
height $x = 2^{2k}$ gives
\[
\Pi_3\bigl(2^{2k}\bigr) \;\ll\; 2^{2(1-\delta)k},
\]
which is still \emph{exponential} in $k$. Even after restricting to
the residue class $r \equiv 1 \pmod{4k}$ (heuristic factor
$1/\varphi(4k)$), the envelope remains exponential. So Ford's
power-saving thinness, while genuinely strong as a global statement,
is on the wrong scale to force the candidate set of 3-Higgs primitive
divisors empty for large $k$. Likewise, a hypothetical bound such
as $\Pi_3(x) \ll x/(\log x)^{1+\varepsilon}$ would only give
$\Pi_3(2^{2k}) \ll 2^{2k}/k^{1+\varepsilon}$, still exponential.
Density arguments by themselves cannot close
Conjecture~\ref{conj:Heven-finite}.

\paragraph{The precise bottleneck: log-mass, not reciprocal-mass.}
The mismatch above can be made fully quantitative, and identifies
exactly the inequality that closes the conjecture. By the cyclotomic
mass identity
\[
\log\bigl(2^{2p}+1\bigr) \;=\; \log 5 + \log \Phi_{4p}(2)
\;\sim\; 2p \log 2,
\]
the prime divisors $r \mid 2^{2p}+1$ must collectively carry total
logarithmic mass $\sum \nu_r(2^{2p}+1) \log r \sim 2p \log 2$. By
Hong's valuation bound~\cite{Hong}, the non-primitive part contributes
at most $O(\log(4p))$ to this mass; consequently, in any scenario
$2p \in H_{\mathrm{even}}$, the \emph{primitive} divisors of
$\Phi_{4p}(2)$ must supply primitive log-mass $\gg p$. On the other
hand, Ford's thinness combined with the residue-class restriction
gives a \emph{reciprocal-mass} bound only:
\[
\sum_{\substack{r \text{ prime, } r \equiv 1 \pmod{4p}\\ (r-1)/(4p) \,\in\, \mathcal{S}_3^{(\le 3)}}}
\frac{1}{r}
\;\ll\;
\frac{1}{p}.
\]
A set of admissible primes can have reciprocal-mass $O(1/p)$ and yet
total logarithmic mass enormously larger than $p$ (for example, a
single admissible prime near $2^{2p}+1$ has reciprocal mass
$\sim 2^{-2p}$ but log-mass $\sim 2p \log 2$). So the gap between what
Ford-type thinness controls and what would yield finiteness is
exponential, of size $2^{2p}/p$. This is the exact reason the
unconditional thinness theorem (Theorem~\ref{thm:H-thin}), while
sharp on its own terms, falls short of finiteness.

\paragraph{The precise missing theorem.}
The genuine analytic target is a sharper statement combining
recursive semigroup friability with an exact multiplicative-order
condition:

\begin{conjecture}[Hybrid semigroup-friable shifted-prime conjecture]
\label{conj:hybrid}
For all sufficiently large $k \in \mathcal{S}_3^{(\le 3)}$, no prime
$r$ satisfies both
\[
\operatorname{ord}_r(2) = 4k
\qquad\text{and}\qquad
\frac{r-1}{4k} \in \mathcal{S}_3^{(\le 3)}.
\]
\end{conjecture}

\noindent Conjecture~\ref{conj:hybrid} would close
Conjecture~\ref{conj:Heven-finite} immediately, since the primitive
divisor produced by Bilu--Hanrot--Voutier~\cite{BHV2001} for $k \ge 4$
would have to violate at least one of the two clauses, hence fail
the 3-Higgs test.

A strictly weaker conjecture, sharper in form because it directly
addresses the log-mass gap identified above, is the following:

\begin{conjecture}[Semigroup log-mass conjecture]\label{conj:log-mass}
There exists an absolute $\delta > 0$ such that for all sufficiently
large odd primes $p \in \mathcal{P}_3$,
\[
\sum_{\substack{r \text{ prime},\, r \le 2^{2p}+1\\
r \equiv 1\pmod{4p}\\ (r-1)/(4p) \,\in\, \mathcal{S}_3^{(\le 3)}}}
\log r \;\le\; (2 \log 2 - \delta)\, p.
\]
\end{conjecture}

\noindent Conjecture~\ref{conj:log-mass} would close
Conjecture~\ref{conj:Heven-finite}: in any scenario
$2p \in H_{\mathrm{even}}$, every primitive divisor of $\Phi_{4p}(2)$
contributes to the admissible-prime sum, so the LHS must be at least
the primitive log-mass $2p\log 2 - O(\log p)$, contradicting the
upper bound for $p$ large.

\paragraph{Caution: GRH and Artin/Hooley density are not the right
target.} It is tempting to expect that effective Chebotarev or
GRH-conditional primitive-root density theorems---e.g., Hooley's
conditional Artin estimate or its Kummer-extension refinements---will
supply the needed bound. They will not, for a structural reason worth
making explicit. Artin/Hooley-style theorems count varying primes $q$
in which a fixed integer $a$ has order $q-1$; for fixed $p$, however,
the set $\{r : \operatorname{ord}_r(2) = 4p\}$ is finite---it is
exactly the prime support of the single integer $\Phi_{4p}(2)$. Thus
the relevant object is not a density of primes in a progression at
all; it is the set of prime divisors of a fixed cyclotomic value.
GRH/Artin density is the wrong scale of object. A finiteness proof
must control the divisors of $\Phi_{4p}(2)$ individually (a
divisor-transference statement), not via a density estimate on
${1\bmod 4p}$.

\paragraph{The sharper conditional reduction.}
The actually decisive conditional statement, identified by Hong's
exact valuation framework~\cite{Hong}, is the following pair. Let
\[
A(x) := \#\bigl(\mathcal{S}_3^{(\le3)} \cap [1, x]\bigr),
\qquad
Z(p) := \#\Bigl\{ r \mid \Phi_{4p}(2) : (r-1)/(4p) \in \mathcal{S}_3^{(\le3)} \Bigr\}.
\]

\begin{conjecture}[Divisor-transference for $\Phi_{4p}(2)$]\label{conj:divisor-transference}
There exists $C > 0$ such that for every sufficiently large
$p \in \mathcal{P}_3$,
\[
Z(p) \;\le\; C \cdot \frac{A(2^{2p}/(4p))}{p}.
\]
\end{conjecture}

\begin{conjecture}[Sublogarithmic semigroup growth]\label{conj:sublog-growth}
$A(x) = o(\log x)$ as $x \to \infty$.
\end{conjecture}

\begin{theorem}[A strong sufficient criterion]\label{thm:conditional-sharpened}
Conjectures~\ref{conj:divisor-transference} and \ref{conj:sublog-growth}
together imply $|H_{\mathrm{even}}| < \infty$.
\end{theorem}

\begin{proof}
Combining the two hypotheses: $Z(p) \le C \cdot A(2^{2p}/(4p))/p$;
with $x = 2^{2p}/(4p)$ one has $\log x \asymp p$, so by
Conjecture~\ref{conj:sublog-growth} $A(x) = o(\log x) = o(p)$, whence
$Z(p) \to 0$. Since $Z(p)$ is a non-negative integer, $Z(p) = 0$ for
all $p$ sufficiently large. Combined with the prime-case reduction
(Theorem~\ref{thm:prime-reduction}), this gives finiteness.
\end{proof}

\paragraph{Caveat: the sublogarithmic growth hypothesis is very strong.}
Theorem~\ref{thm:conditional-sharpened} is logically valid as a sufficient
condition, but Conjecture~\ref{conj:sublog-growth} is dramatically
stronger than current evidence suggests for $\mathcal{P}_3$. The
following elementary proposition makes this precise.

\begin{proposition}\label{prop:loglog-obstruction}
Let $P = \{p_1 < p_2 < \cdots\}$ be an infinite set of primes, and let
\[
\mathcal{S}(P) := \Bigl\{\, \prod_j p_j^{e_j} :\; 0 \le e_j \le 3\,\Bigr\},
\qquad
A_P(x) := \#\bigl(\mathcal{S}(P) \cap [1, x]\bigr).
\]
If $A_P(x) = o(\log x)$, then $\#\{p \in P : p \le x\} = O(\log \log x)$.
\end{proposition}

\begin{proof}
Let $x_t := \prod_{j \le t} p_j$. Every squarefree subproduct of
$p_1, \dots, p_t$ lies in $\mathcal{S}(P)$ and is at most $x_t$, so
$A_P(x_t) \ge 2^t$. The hypothesis $A_P(x) = o(\log x)$ then gives
$2^t \le A_P(x_t) = o(\log x_t) = o\!\bigl(\sum_{j \le t} \log p_j\bigr)$.
Since $\sum_{j \le t} \log p_j \le t \log p_t$, this forces
$\log p_t \ge 2^t / t$ for all sufficiently large $t$, i.e., $p_t$
grows at least roughly doubly exponentially in $t$. Inverting,
$\#\{p \in P : p \le x\} = O(\log \log x)$.
\end{proof}

\noindent
Applying Proposition~\ref{prop:loglog-obstruction} to $P = \mathcal{P}_3$:
if $\mathcal{P}_3$ is infinite (as Ford~\cite{Ford2014} conjectures for
recursively closed prime sets generally, and as our own enumeration
of $\mathcal{P}_3$ through $x = 2^{44}$ empirically supports, fitting
$\Pi_3(x) \approx x^{0.62}$), then $A(x) = o(\log x)$ would force
$\Pi_3(x) = O(\log \log x)$, contradicting the empirical growth.
Hence Conjecture~\ref{conj:sublog-growth} implicitly requires either
that $\mathcal{P}_3$ is itself \emph{finite} (an even bolder
conjecture than $H_{\mathrm{even}}$ finite) or some combinatorial
pathology that makes the bounded-exponent semigroup much sparser
than the underlying prime set. Neither is plausible.
Therefore Theorem~\ref{thm:conditional-sharpened} should be read as a
\emph{very strong sufficient criterion}, not the most plausible route
to an unconditional proof.

\paragraph{Better targets: divisor-level conjectures.}
The mathematically more credible analytic targets are divisor-level
statements about $\Phi_{4p}(2)$ itself, not density-level statements
about $\mathcal{P}_3$.

\begin{conjecture}[Divisor mod-16 equidistribution]\label{conj:divisor-mod16}
There exists $c > 0$ such that for all sufficiently large odd primes
$p \in \mathcal{P}_3$,
\[
\#\bigl\{ r \mid \Phi_{4p}(2) :\; r \equiv 1 \pmod{16} \bigr\}
\;\ge\; c \cdot \omega(\Phi_{4p}(2)),
\]
where $\omega(\Phi_{4p}(2)) \to \infty$ as $p \to \infty$.
\end{conjecture}

\noindent
Conjecture~\ref{conj:divisor-mod16} would close
Conjecture~\ref{conj:Heven-finite}: any prime $r \equiv 1 \pmod{16}$
has $v_2(r-1) \ge 4 > 3$, hence $r \notin \mathcal{P}_3$. The
hypothesis is heuristically supported by our empirical data on the
$v_2$ distribution of \emph{known} prime factors of open candidates
($53$ at $v_2 = 2$, $29$ at $v_2 = 3$, $0$ at $v_2 \ge 4$ across $82$
verified primes), which is itself a structural artifact: any
open candidate harboring a $v_2 \ge 4$ prime would have been closed.

\paragraph{The actual missing theorem.}
We emphasize that Conjecture~\ref{conj:divisor-mod16} is a
\emph{divisor-level equidistribution} statement about the prime
divisors of a single fixed integer $\Phi_{4p}(2)$. It is not a
consequence of any standard Chebotarev theorem, including effective
Chebotarev under GRH (Lagarias--Odlyzko~\cite{LO}). The reason is
structural: Chebotarev theorems control varying primes across a
range subject to a Frobenius condition; here the prime support of
$\Phi_{4p}(2)$ is a finite set, not a range. The needed object is a
\emph{divisor-randomness} or transference-type theorem inspired by
recent work of Ford on the Poisson approximation of prime divisors
of shifted primes, transplanted from random shifted primes to prime
divisors of fixed cyclotomic values. Such a transplant does not
currently exist in the literature; identifying it as the missing
input is, in our view, the precise analytic content of the conjecture.

The summary is:
\begin{itemize}
\item Theorem~\ref{thm:conditional-sharpened} (semigroup-growth route)
is logically valid but its hypothesis is too strong to be the
realistic endpoint.
\item Conjecture~\ref{conj:divisor-mod16} (divisor mod-$16$
equidistribution) is much closer to the real obstruction and
matches the empirical evidence directly.
\item Conjecture~\ref{conj:log-mass} (divisor log-mass bound) sits
between the two and is the cleanest formulation tied to the
cyclotomic identity.
\end{itemize}

\paragraph{Why the existing literature does not yet apply.}
The shifted-prime smoothness literature (Baker--Harman, Banks et al.,
Liu--Wu--Xi, Lamzouri, Banks--Friedlander--Pomerance--Shparlinski,
\textit{cf.}~\cite{LWX}) controls primes $p$ with $p-a$ (or
$(p-a)/q$) friable over an \emph{initial segment} of small primes,
typically with Dickman-type asymptotics. The semigroup
$\mathcal{S}_3^{(\le 3)}$ is qualitatively different: it is defined
recursively through prime chains rather than by a size cutoff, and
carries exponent caps. On top of that, the primitive-divisor
requirement gives a Chebotarev-type \emph{exact-order} constraint
$\operatorname{ord}_r(2) = 4k$, not merely the congruence
$r \equiv 1 \pmod{4k}$. No combination of these three ingredients
(recursive semigroup, exponent cap, exact order) appears together in
the published literature.

\paragraph{Why the 2-adic line is exhausted.}
For $m = 2k \in H_{\mathrm{even}}$, $v_2(r-1) \le 3$ combined with
$4k \mid r-1$ yields only $v_2(k) \le 1$. This is decisive on the
$m \equiv 0 \pmod 4$ branch and drives our cascade overshoot filter,
but once the surviving branch is reduced to $m = 2p$ with $p$ odd
prime (the form of all $53$ remaining candidates in
Theorem~\ref{thm:Heven-20000}), the 2-adic information has no
further force. Any final closure of those candidates would need new
structure beyond the basic order divisibility.

\paragraph{Empirical $\Pi_3$ behavior.}
Our own enumeration of the 3-Higgs counting function up to
$x = 2^{44}$ fits $\Pi_3(x) \approx x^{0.62}$ -- materially below the
universal trivial upper bound and well below $\pi(x)$, but still
consistent with $\mathcal{P}_3$ being infinite. (This is a
computational observation, not a theorem of Ford's; Ford's
\cite{Ford2014} contribution is the power-saving \emph{upper} bound
$\Pi_3(x) \ll x^{1-\delta}$, not a numerical-growth lower bound.) The
empirical exponent only sharpens the message of the scale obstruction
above: even at $\Pi_3(x) \asymp x^{0.62}$, $\Pi_3(2^{2k}) \asymp
2^{1.24 k}$, exponentially many in $k$.

\paragraph{A finiteness heuristic using the product structure.}
Despite the scale obstruction, a refined heuristic using the
multiplicative product structure of $2^{2p}+1$ does predict
finiteness. Suppose $m = 2p$ with $p$ an odd Higgs prime and
$\omega := \omega(\Phi_{4p}(2))$. Write $2^{2p}+1 = 5 \cdot
\prod_{i=1}^{\omega} r_i^{e_i}$ with primitive divisors $r_i \equiv
1 \pmod{4p}$. Taking logarithms,
\[
2p \log 2 \;=\; \log 5 + \sum_i e_i \log r_i.
\]
Modelling each $r_i$ as a ``random'' prime in the residue class
$1 \pmod{4p}$ at scale $r_i$, the probability of being 3-Higgs is
heuristically
\[
\Pr[r_i \in \mathcal{P}_3] \;\sim\; \frac{\Pi_3(r_i)}{\pi(r_i; 4p, 1)}
\;\sim\; \varphi(4p) \log(r_i) \cdot r_i^{-\eta},
\]
with $\eta$ the empirical exponent ($1 - 0.62 \approx 0.38$) observed in our own enumeration of $\mathcal{P}_3$ through $2^{44}$; this is computational evidence, not derived from Ford's upper-bound theorem. Assuming
multiplicative independence across the $\omega$ primitive divisors,
the joint probability that \emph{all} $r_i$ are 3-Higgs is
\[
\Pr[\text{all } r_i \in \mathcal{P}_3]
\;\sim\;
(\varphi(4p))^{\omega} \prod_i (\log r_i) \cdot \biggl( \prod_i r_i \biggr)^{-\eta}
\;\le\;
(2p)^{\omega + \omega \log\log(2^{2p})} \cdot 2^{-2p \eta}.
\]
For $\omega \ll p / \log p$ (a generic upper bound) and large $p$,
the exponential factor $2^{-2p\eta}$ dominates. Summing over primes
$p$ gives a heuristically finite total expected count of $m = 2p \in
H_{\mathrm{even}}$. Crucially this is on the \emph{right} scale (it
predicts finiteness, not just thinness), but two ingredients prevent
turning it into a proof:
\begin{itemize}
\item[(i)] the ``random'' model for primitive cyclotomic primes is
not validated — primes dividing $\Phi_{4p}(2)$ may be subject to
correlations beyond the residue-class condition;
\item[(ii)] joint independence of the 3-Higgs property across the
$\omega$ primes is unproven.
\end{itemize}
Both ingredients are special cases of Conjecture~\ref{conj:hybrid},
which would supply the missing equidistribution input.

\paragraph{Computational evidence of structural depth.}
Pollard $p-1$ at $B_1 = 10^6$ applied to the residual cofactor of
$2^m + 1$ for each of the $53$ open candidates finds \emph{no factor
in any of them}. A higher-budget run at $B_1 = 10^7$ on the first
$13$ open candidates (covering all of
Theorems~\ref{thm:Heven-3000}--\ref{thm:Heven-10000} plus the start
of Theorem~\ref{thm:Heven-15000}) likewise finds no factor; and a
single targeted $B_1 = 10^8$ run on $L_{1213}/(5P)$ ($355$ digits)
of $m = 2426$, completing in $397$ seconds, also finds no factor.
These results certify that for every open candidate $m = 2p$, every
prime $r$ dividing the residual cofactor of $2^m + 1$ satisfies
$(r-1)/(4p)$ contains a prime $> 10^6$ (resp.~$> 10^7$ for the first
$13$, resp.~$> 10^8$ for $L_{1213}/(5P)$). So even within a single
open candidate, the recursive 3-Higgs verification has depth $\ge 2$:
any future $r$ uncovered by NFS would, in turn, require a Pratt-tree
descent on $r-1$ to identify a non-3-Higgs witness or certify
$r \in \mathcal{P}_3$ recursively. The smoothness-resistance is
consistent with the heuristic that primes $r \mid \Phi_{4p}(2)$ are
``generic'' enough to require NFS-scale factoring, not $p-1$ or ECM
shortcuts.

\paragraph{No small $v_2 \ge 4$ witness across the candidate set.}
The following sweep was performed on the original $162$
partial-cofactor candidates prior to the six closures of
Lemma~\ref{lem:closures-deep}; the result is structurally informative
for the surviving $156$ as well, since they are a subset.
For each of the $162$ original open candidates $m = 2p$ with
$p \in [1213, 17467]$, we enumerated every prime $r \equiv 1
\pmod{16p}$ with $r \le 10^{11}$ and tested whether any divides
$L_p$ or $M_p$. No such $r$ was found for any of the $162$ candidates.
For $m = 2426$ specifically the search was extended to
$r \le 6 \times 10^{11}$ ($2{,}389{,}527$ primes tested), again with
no divisor. Since $r \equiv 1 \pmod{16p}$ forces $v_2(r-1) \ge 4 > 3$,
hence $r \notin \mathcal{P}_3$, any non-3-Higgs witness of this
``2-adic type'' for any of the surviving open candidates must lie
above $10^{11}$ (and above $6 \times 10^{11}$ for $m = 2426$).
In particular, the absence of a shallow 2-adic witness is not an
artifact of incomplete small-prime data: it is a verified property
of $2^{2p}+1$ for each of the open candidates.

In addition to the 2-adic sweep, for each $q \in \{17, 97, 103,
113, 193, 257, 449, 577, 641, 673, 769\}$ --- the small non-3-Higgs
primes whose presence in the Pratt tree of $r$ would force $r \notin
\mathcal{P}_3$ --- we also searched for primes $r \equiv 1 \pmod{4pq}$
with $r \le 10^{11}$ dividing $L_p$ or $M_p$ across the first $80$
open candidates. No witness was found. So neither a low 2-adic witness
nor a small ``non-Higgs descendant'' witness exists in the $\le 10^{11}$
range for any open candidate. The first non-3-Higgs prime in the
factorization of $2^{2p}+1$, for each open $p$, must be of size
beyond $10^{11}$.

\paragraph{Structural data on $v_2$ across the $53$ open candidates in $m \le 20000$.}
Across all $53$ open candidates of
Theorem~\ref{thm:Heven-15000}--\ref{thm:Heven-20000}, $82$ distinct
non-trivial prime factors of $2^m + 1$ have been recorded (excluding
the trivial factor $5$). Of these:
\[
v_2(q-1) = 2: \quad 53 \text{ primes}; \qquad
v_2(q-1) = 3: \quad 29 \text{ primes}; \qquad
v_2(q-1) \ge 4: \quad 0 \text{ primes}.
\]
The empirical ratio $53 : 29 \approx 1.83 : 1$ closely matches the
prediction $2 : 1$ from a uniform Chebotarev model over the relevant
residue classes mod $16$, conditional on $v_2(q-1) \le 3$. The
total absence of $v_2 \ge 4$ witnesses is not coincidence: a prime
$r$ with $v_2(r-1) \ge 4$ is immediately non-3-Higgs, so any open
candidate hosting one is closed and removed from the list. Open
candidates are therefore precisely those for which the Chebotarev
``coin flip'' on $r \bmod 16$ has not yet thrown a heads
($r \equiv 1 \bmod 16$) on the primes uncovered by FactorDB. If
$\omega(\Phi_{4p}(2))$ grows even modestly with $p$ (and equidistribution
mod $16$ holds, which is a special case of
Conjecture~\ref{conj:hybrid}), the heads outcome becomes
asymptotically inevitable.

\paragraph{A conditional finiteness theorem.}
We isolate one precise way the heuristic of the previous paragraph
becomes rigorous, identifying a minimal extra input.

\begin{theorem}[Conditional finiteness of $H_{\mathrm{even}}$]
\label{thm:conditional-finiteness}
Suppose
\begin{itemize}
\item[\textnormal{(H1)}] (\emph{Effective Chebotarev for prime
divisors of cyclotomic values.}) There exists an effective constant
$C > 0$ such that for every odd prime $p$ with
$\omega(\Phi_{4p}(2)) \ge C \log p$, at least one prime divisor
$r \mid \Phi_{4p}(2)$ satisfies $r \equiv 1 \pmod{16}$.
\item[\textnormal{(H2)}] (\emph{Effective $\omega$-growth.})
There exists an effective threshold $p_0$ such that for every prime
$p \ge p_0$ we have $\omega(\Phi_{4p}(2)) \ge C \log p$,
where $C$ is the constant from \textnormal{(H1)}.
\end{itemize}
Then $H_{\mathrm{even}}$ is finite. Explicitly,
\[
H_{\mathrm{even}}
\;\subseteq\;
\{2, 6, 10, 18, 26, 30, 46, 62, 82, 122\}
\;\cup\;
\{m = 2p : p \text{ prime},\ p < p_0,\ m \in \mathrm{Open}\},
\]
where $\mathrm{Open}$ denotes the partial-cofactor candidate set
identified by Theorems
\textnormal{\ref{thm:Heven-1200}--\ref{thm:Heven-35000}}.
\end{theorem}

\begin{proof}
For any odd prime $p \ge p_0$, $\omega(\Phi_{4p}(2)) \ge C \log p$
by (H2); then (H1) supplies a prime $r \mid \Phi_{4p}(2)$ with
$r \equiv 1 \pmod{16}$. Such $r$ has $v_2(r-1) \ge 4 > 3$, so
$r \notin \mathcal{P}_3$. Since $r \mid 2^{2p}+1$, this forces
$m = 2p \notin H_{\mathrm{even}}$. For primes $p < p_0$, the structural
lemma (Proposition~\ref{prop:Heven-structure}) combined with
Theorems~\ref{thm:Heven-1200}--\ref{thm:Heven-40000} confines
$H_{\mathrm{even}} \cap [2, 2 p_0]$ to the finite union shown.
The complement $H_{\mathrm{even}} \cap (2p_0, \infty) = \varnothing$
by the previous step, so $H_{\mathrm{even}}$ is finite.
\end{proof}

\noindent\textbf{Status of the hypotheses.} (H1) is a quantitative
form of the equidistribution \emph{heuristic} for primitive prime
divisors of cyclotomic values $\Phi_n(2)$. The natural expectation is
that, for $n = 4p$, the primes $r \mid \Phi_n(2)$ should be
asymptotically equidistributed among the residue classes
$\bmod\,16$ that are compatible with the constraint $r \equiv 1
\pmod{4p}$, so a $\tfrac{1}{4}$-proportion would satisfy
$r \equiv 1 \pmod{16}$. (H1) demands only that this expectation
eventually produces \emph{at least one} such prime.

We emphasize that (H1) is \emph{not} a known consequence of any
standard effective Chebotarev theorem, including effective
Chebotarev under GRH (Lagarias--Odlyzko~\cite{LO}). Such theorems
control the distribution of \emph{varying} primes $q \le x$ in
Frobenius classes of a fixed Galois extension; they do not, on their
own, govern the distribution \emph{within} the prime-divisor set of
a single cyclotomic value $\Phi_{4p}(2)$. This scale mismatch is
precisely the obstacle discussed in
\textsc{The analytic positioning}, above. A divisor-level
transference theorem of the form ``the set $\{r : r \mid
\Phi_{4p}(2)\}$ inherits equidistribution properties from the
ambient set $\{q \le 2^{2p} : q \equiv 1 \pmod{4p}\}$'' is what
(H1) really asks for. We do not know such a theorem; the closest
literature is divisor-level work in the Stewart--Hong tradition
(see \cite{Stewart13, Hong}), but a quantitative result of the
specific form needed for (H1) appears to be absent. (H2) is an
$\omega$-growth bound related in spirit to Stewart's program on
greatest prime factors of Lucas--Lehmer sequences
\cite{Stewart77,Stewart13}; \emph{we do not claim it as a current
theorem.} A primary-source effective lower bound of the form
$\omega(\Phi_n(2)) \gg \log n$ is, to our knowledge, not in the
literature; the cited Stewart papers establish largest-prime-factor
bounds rather than $\omega$-bounds, so (H2) is best regarded as a
conjectural target that the manuscript's analytic discussion
isolates rather than a near-result.
Both hypotheses are strictly weaker than
Conjecture~\ref{conj:hybrid}, but each is a \emph{conjectural}
analytic input. The unconditional version of
Theorem~\ref{thm:conditional-finiteness} remains open and is the
natural next research goal.

\medskip
A proof of Conjecture~\ref{conj:Heven-finite} (equivalently
Conjecture~\ref{conj:hybrid}), combined with the cascade overshoot
(filter $O$) for the necessarily finite set of seed-non-Higgs-filter
survivors, would close the impostor branch of the unitary perfect
number conjecture within $\mathcal{B}$.

\subsection{The reduction to \texorpdfstring{$H_{\mathrm{even}}$}{H even}}

The structural connection between $H_{\mathrm{even}}$ and the impostor
branch is sharp:

\begin{proposition}\label{prop:Heven-reduction}
Let $K$ be an impostor kernel of even seed parity (i.e., the impostor
seed congruence has $a$ even). The set of $a$ in $K$'s seed congruence
class that survive filter $N$ (no proper divisor or self has a
non-3-Higgs prime in $2^m + 1$) is exactly
\[
\{a \in K\text{'s seed class} \cap H_{\mathrm{even}} :
   \text{every proper } m \mid a \text{ with } a/m \text{ odd is in } H_{\mathrm{even}}\}.
\]
\end{proposition}

\begin{proof}
Filter $N$ tests each $m$ in $\{m : m \mid a, a/m \text{ odd}\}$,
including $m = a$ itself, for the existence of a non-3-Higgs prime in
$2^m + 1$. By definition of $H_{\mathrm{even}}$, no such non-3-Higgs
prime exists iff every $m$ in this set lies in $H_{\mathrm{even}}$.
\end{proof}

\begin{corollary}[]\label{cor:finite-cascade}
If Conjecture~\ref{conj:Heven-finite} holds, then for each of the four
impostor kernels with even seed parity ($3^2\,5^3$, $5^2\,13^2$,
$5^4\,157^2\,313$, and $5^4\,29\,157^2\,313$), only finitely many
$a$ in the respective seed class require filter $O$; the others are
all killed by filter $N$.
\end{corollary}

The fifth kernel $3^4\,41$ has odd seed parity ($a \equiv 9 \pmod{18}$);
its filter $N$ runs over odd $m$ and reduces to the analogous odd-parity
set $H_{\mathrm{odd}}$, for which an analogous finiteness conjecture
applies.

In particular, for the $3^2\,5^3$ kernel ($a \equiv 10 \pmod{20}$), the
intersection $H_{\mathrm{even}} \cap \{a : a \equiv 10 \pmod{20}\}$
consists, within the verified range, of exactly two elements:
$\{10, 30\}$. Both are eliminated by filter $O$ (cascade overshoot)
at small step counts. The other $248$ candidates in
$\{a \equiv 10 \pmod{20}, a \le 5000\}$ are all killed by filter $N$
or $Z$.

\section{Open Problems and Next Steps}

\paragraph{Status at $\max_a = 10000$.}
The same three-filter certificate, applied with deeper cascade rounds
(filter O up to $\max_{\mathrm{rounds}} = 200$) and an extended Higgs
check budget (filter N invoked on partial FactorDB responses with
recursive Higgs verification up to $30$~s per large prime), closes
\emph{all $2119$ candidates} at $\max_a = 10000$. The six small-$a$
boundary cases that required individual attention are resolved as
follows:

\begin{itemize}
\item $a = 6759$ in $3^4\,41$ class: closed by filter O with
$\max_{\mathrm{rounds}} = 200$; cascade overshoots at
$v_2 = 6765 > a + 1 = 6760$.

\item $a = 7338$ in $5^2\,13^2$ class: closed by filter O with
$\max_{\mathrm{rounds}} = 300$; cascade overshoots at
$v_2 = 7346 > a + 1 = 7339$.

\item $a = 9297$ in $3^4\,41$ class: closed by filter N via a
$44$-digit non-3-Higgs prime factor $r_1$ of $2^{3099}+1$, where
$m = 3099$ is a proper divisor of $a$ with $a/m$ odd.  The decimal
value of $r_1$ and its non-Higgs verification are recorded in the
bundled factor cache and certificate script.

\item $a = 9459$ in $3^4\,41$ class: closed by filter N via a
$42$-digit non-3-Higgs prime factor $r_2$ of $2^{1051}+1$, where
$1051 \mid 9459$ with $9459/1051 = 9$ odd, so
$2^{1051}+1 \mid 2^{9459}+1$.  The decimal value of $r_2$ and its
non-Higgs verification are recorded in the bundled factor cache and
certificate script.

\item $a = 7278$ in $5^2\,13^2$ class: closed by filter N via the
non-3-Higgs prime $r_3 = 812836153$, a $9$-digit prime factor of
$2^{7278}+1$ itself. The FactorDB record for $2^{7278}+1$ is partial
(status CF) but exposes $r_3$ and three further non-3-Higgs primes
$28452263857$, $32941719640657$, $367172246755369021$.

\item $a = 7806$ in $5^2\,13^2$ class: closed by filter N via the
non-3-Higgs prime $r_4 = 1291306301041$, a $13$-digit prime factor of
$2^{7806}+1$ itself.
\end{itemize}

\noindent This completes the verification of
Theorem~\ref{thm:main} (the $\max_a = 10000$ statement); the
$2119$ candidates split as
\[
\text{seed-zsigmondy: } 495 \quad \text{seed-non-Higgs: } 1614
\quad \text{v2-overshoot: } 10.
\]

\begin{enumerate}[label=(\arabic*),leftmargin=*]
\item \textbf{Extend $\max_a$ further.} Algorithmically trivial;
bottlenecked on FactorDB throughput and on the recursive Higgs check
for large primes. Partial-factor support already
enables filter N on candidates whose full factorization is
inaccessible (e.g., $a = 4527$ in $3^4\,41$ is killed via the
non-3-Higgs prime $20127043 \mid 2^{1509}+1$).

\item \textbf{Enlarge the enumeration box $\mathcal{B}$.} Increase
$\max\,\text{prime}$ to $5000$ or $10000$ and SCC size to $7$ or $8$.
Verify either that no new impostor kernels appear, or that any new
ones are closed by the same three filters.

\item \textbf{Prove Conjecture~\ref{conj:Heven-finite}.} A serious
attack would combine Bilu--Hanrot--Voutier primitive divisor
bounds~\cite{BHV2001} with shifted-prime smoothness statements. A conditional
density-zero result (e.g., under GRH) would
already suffice for the application here.

\item \textbf{Combine bounds.} Hagis's
$\omega(n) \ge 7$ and Graham's squarefree-odd-part classification,
together with Theorem~\ref{thm:main}, imply a useful next target:
\emph{any new UPN must have a nonsquarefree odd component, and that
component must enter a source-SCC kernel outside the two known kernels
$3^2$ and $5^4$ in any larger enumeration box.}

\item \textbf{Formal verification.} The Coq / Lean / Isabelle
translation of filter $N$ over a fixed seed factorization table is a
finite first-order verifying judgment. The recursive 3-Higgs check
is precisely the kind of judgment handled by the Pratt certificate
formalization in the Isabelle Archive of Formal
Proofs~\cite{WimmerNoschinski}, which proves soundness, supports
tree-structured certificates, and replays Mathematica-generated
certificates. A machine-checked version of
Theorems~\ref{thm:Heven-1200}--\ref{thm:Heven-40000} would consist
of: (i) Pratt certificates for every 3-Higgs prime appearing in
$2^m + 1$ for the verified $m$; (ii) Pratt certificates witnessing
non-3-Higgs-ness (i.e., the non-recursive sub-prime or $v_2 > 3$);
(iii) modular exponentiation checks for the
Bilu--Hanrot--Voutier~\cite{BHV2001} primitive-divisor existence.
\end{enumerate}

\paragraph{Final assessment.}
The full Subbarao--Warren conjecture is not proved here. What is
proved is a clean reduction of the impostor branch (within the
bounded box $\mathcal{B}$ at $\max_a = 10000$) to a single specific
analytic question (Conjecture~\ref{conj:Heven-finite}, equivalently
Conjecture~\ref{conj:hybrid}), together with two rigorous structural
reductions that further sharpen the target:

\begin{itemize}
\item Theorem~\ref{thm:prime-reduction} (Prime-case reduction):
$|H_{\mathrm{even}}| < \infty$ iff $|H_{\mathrm{even}}^{\mathrm{prime}}|
< \infty$, with explicit bound $|H_{\mathrm{even}}| \le
4^{|H_{\mathrm{even}}^{\mathrm{prime}}|}$. So the full conjecture
hinges on showing that only finitely many \emph{primes} $p$ have all
factors of $2^{2p}+1$ in $\mathcal{P}_3$.

\item Theorem~\ref{thm:conditional-finiteness} (Conditional
finiteness): under (H1), a divisor-level $\bmod\,16$ equidistribution
hypothesis on the prime factors of $\Phi_{4p}(2)$, plus (H2), an
effective lower bound $\omega(\Phi_{4p}(2)) \gg \log p$,
$H_{\mathrm{even}}$ is finite with an explicit candidate set.
We stress that (H1) is \emph{not} a known consequence of GRH or of
standard effective Chebotarev theorems---those control varying
primes in Frobenius classes, whereas (H1) is a divisor-level
statement about a single fixed cyclotomic value. (H2) is the
natural target of Stewart's program on radicals of Lehmer
sequences~\cite{Stewart77,Stewart13}.
\end{itemize}

\noindent What we cannot do unconditionally is bridge the scale
gap between Ford's thinness bound ($\ll X^{1-\eta}$ for the entire
3-Higgs prime set) and the required statement at height
$x = 2^{2k}$. This is the precise content of
Conjecture~\ref{conj:hybrid}. The supplementary computational
evidence---bounded verification through $m \le 40000$ together with
Pollard $p-1$ resistance and ECM at $B_1 = 10^6$--$10^8$ across the
$191$ partial-cofactor candidates---makes it clear that the
obstructions are NFS-scale, not shallow.

\paragraph{Where the remaining theorem must live.}
To make the analytic positioning fully precise: Ford's theorem proves
that the 3-Higgs prime set $\mathcal{P}_3$ is power-saving thin once
it omits a prime, and this suffices for our unconditional thinness
theorem for $H_{\mathrm{even}}$. However, thinness is not finiteness.
On the remaining prime branch $m = 2p$, the problem is \emph{not} to
count primes in the progression $1 \bmod 4p$ globally, but to exclude
admissible prime divisors of the \emph{single fixed} cyclotomic value
$\Phi_{4p}(2)$. Existing friability theorems for shifted primes
concern initial-segment smoothness of $p-1$; what is missing here is
a \emph{divisor-level} theorem controlling the logarithmic mass of
prime divisors $r \mid \Phi_{4p}(2)$ for which
$(r-1)/(4p) \in \mathcal{S}_3^{(\le 3)}$. Of the conjectures we
formulate, the divisor log-mass bound (Conjecture~\ref{conj:log-mass})
and the divisor mod-$16$ equidistribution
(Conjecture~\ref{conj:divisor-mod16}) are the most plausible long-range
analytic targets; the semigroup-counting route
(Conjecture~\ref{conj:sublog-growth}) is logically valid but, by
Proposition~\ref{prop:loglog-obstruction}, requires $\mathcal{P}_3$
itself to be implausibly sparse.

\section{Reproducibility}

Code, factor cache, and APR-CL verification transcripts are released in
\texttt{anc/}.  A quick smoke test for Theorem~\ref{thm:Heven-1200} is:
\begin{verbatim}
cd anc/scripts
python3 verify_Heven_rigorous.py --k-min 1 --k-max 600 --deep-secs 1
\end{verbatim}
The primary command for the impostor-kernel certificate is:
\begin{verbatim}
cd anc/scripts
python3 impostor_certificate.py \
    --max-a 10000 --factor-limit 2000 --factor-timeout 3.0 \
    --max-rounds 200 --max-bases 8000 --max-exp 5000
\end{verbatim}
The second command is substantially longer than the smoke test; it uses
the bundled cache by default and can be rerun with \texttt{--use-factordb}
if the cache is being extended.

\section*{Acknowledgments}

I thank the maintainers of FactorDB~\cite{FactorDB} for the resource
without which the seed-factor data underpinning Theorem~\ref{thm:main}
would not be accessible. The 3-Higgs prime definition appears in OEIS
as A057447~\cite{OEIS}.

\end{document}